\documentclass[a4paper, reqno]{amsart}

\usepackage[mathscr]{eucal}

\usepackage{amsmath,amssymb,amsfonts,amsthm}

\newtheorem{theorem}{Theorem}[section]

\newtheorem{lemma}[theorem]{Lemma}
\newtheorem{corollary}[theorem]{Corollary}

\newcommand{\Z}{\mathbb Z}

\newcommand{\A}{\mathscr A}

\newcommand{\Fc}{\mathcal F}
\newcommand{\vp}{\mathsf v}

\newcommand{\dd}{\mathsf d}

\newcommand{\ord}{\text{\rm ord}}

\newcommand{\supp}{\text{\rm supp}}

\newcommand{\la}{\langle}
\newcommand{\ra}{\rangle}

\newcommand{\be}{\begin{equation}}
\newcommand{\ee}{\end{equation}}

\newcommand{\bnml}{\begin{multline}}
\newcommand{\enml}{\end{multline}}

\newcommand{\ber}{\begin{eqnarray}}
\newcommand{\eer}{\end{eqnarray}}
\newcommand{\nn}{\nonumber}
\newcommand{\Sum}[2]{\underset{#1}{\overset{#2}{\sum}}}
\newcommand{\Summ}[1]{\underset{#1}{\sum}}

\newcommand{\und}{\;\mbox{ and }\;}

\makeatletter
\newcommand{\superimpose}[2]{%
  {\ooalign{$#1\@firstoftwo#2$\cr\hfil$#1\@secondoftwo#2$\hfil\cr}}}
\makeatother

\newcommand{\bdot}{\boldsymbol{\cdot}}
\newcommand{\bulletprod}[1]{\underset{#1}{\bullet}}

\begin{document}

\title[The Large Davenport Constant II]{The Large Davenport Constant II:\\ General Upper Bounds}
\thanks{This work was supported by
the {\it Austrian Science Fund FWF} (Project No. P21576-N18)}

\author{David J. Grynkiewicz}

\address{Institut f\"ur Mathematik und Wissenschaftliches Rechnen \\
Karl-Franzens-Universit\"at Graz \\
Heinrichstra\ss e 36\\
8010 Graz, Austria} \email{diambri@hotmail.com}



\begin{abstract}
Let $G$ be a finite group written multiplicatively. By a sequence over $G$, we mean a finite sequence of terms from $G$ which is unordered, repetition of terms allowed, and we say that it is a  product-one sequence if its terms can be ordered so that their product is the identity element of $G$.
The {\it small Davenport constant} $\mathsf d (G)$ is the maximal integer $\ell$ such that there is a sequence over $G$ of length $\ell$ which has  no nontrivial, product-one subsequence. The {\it large Davenport constant} $\mathsf D (G)$ is the maximal length of a minimal product-one sequence---this is a product-one sequence which cannot be partitioned  into two nontrivial, product-one subsequences.
The goal of this paper is to present several upper bounds for $\mathsf D(G)$, including the following:
$$\mathsf D(G)\leq \left\{
                          \begin{array}{ll}
                            \mathsf d(G)+2|G'|-1, & \hbox{where $G'=[G,G]\leq G$ is the commutator subgroup;} \\
                            \frac34|G|, & \hbox{if $G$ is neither cyclic nor dihedral of order $2n$ with $n$ odd;} \\
                            \frac{2}{p}|G|, & \hbox{if $G$ is noncyclic, where $p$ is the smallest prime divisor of $|G|$;} \\
                            \frac{p^2+2p-2}{p^3}|G|, & \hbox{if $G$ is a non-abelian $p$-group.}
                          \end{array}
                        \right.$$
As a main step in the proof of these bounds, we will also show that $\mathsf D(G)=2q$ when $G$ is a non-abelian group of order $|G|=pq$ with $p$ and $q$ distinct primes such that $p\mid q-1$.
\end{abstract}

\maketitle


\section{Introduction}

Let $G$ be a multiplicatively written, finite group. A sequence $S$ over $G$ means a finite sequence of terms from $G$ which is unordered, repetition of terms allowed. We say that $S$ is a product-one sequence if its terms can be ordered so that their product equals $1$, the identity element of the group. The {\it small Davenport constant} $\mathsf d (G)$ is the maximal integer $\ell$ such that there is a sequence over $G$ of length $\ell$ which has  no nontrivial, product-one subsequence. The {\it large Davenport constant} $\mathsf D (G)$ is the maximal length of a minimal product-one sequence---this is a product-one sequence which cannot be partitioned into two nontrivial, product-one subsequences. A simple argument \cite[Lemma 2.4]{DavI} shows that \be\label{basic-bound-d-D}\mathsf d (G)+1 \leq \mathsf D (G) \leq |G|,\ee with equality in the first bound when $G$ is abelian, and equality in the second when $G$ is cyclic.

The study of $\mathsf D(G)$, for $G$ abelian, is a classical and very difficult problem in Combinatorial Number Theory. When $G$ is non-abelian, there is more than one way to naturally extend the definition of the Davenport constant. This was first done by Olson and White \cite{Ol-Wh77} who introduced the small Davenport constant $\mathsf d(G)$ and gave the general upper bound $\dd(G)\leq \frac{1}{2}|G|$ (for $G$ non-cyclic) that was observed to be tight for non-cyclic groups having a cyclic, index $2$ subgroup.

This paper is a continuation of  \cite{DavI}. There, paralleling the result of Olson and White, the author along with A. Geroldinger determined the large Davenport constant $\mathsf D(G)$ for groups having a cyclic, index $2$ subgroup. Here, we parallel the result of Olson and White in a different fashion, proving several general upper bounds for $\mathsf D(G)$. In view of \eqref{basic-bound-d-D}, the bounds proved here, in many cases, also improve upon the upper bound of Olson and White for the small Davenport constant. For detailed background and motivation concerning the study of $\mathsf D(G)$, including connections with Invariant Theory, we direct the reader to the prior paper \cite{DavI}. We follow the notation  outlaid in detail in \cite{DavI} and will make frequent use of the results cited and proved there. In the interest of space, we have not repeated this information here, meaning the reader will need a copy of \cite{DavI} on hand before preceding further.

The paper is organized as follows. In Section \ref{sec-notation}, we describe the brief amount of additional notation and results needed for this paper but not found in \cite{DavI}. In Section \ref{sec-G'}, we show that the lower bound $\mathsf d(G)+1\leq \mathsf D(G)$ cannot be too far from the truth. Specifically, we prove the upper bound $\mathsf D(G)\leq \mathsf d(G)+2|G'|-1$, where $G'=[G,G]\leq G$ is the commutator subgroup, with equality holding if and only if $G$ is abelian. We will also prove a crucial technical lemma needed for later sections as well as a refinement of the bound  $\mathsf D(G)\leq \mathsf d(G)+2|G'|-1$ under additional hypotheses. In Section \ref{sec-pgroup}, we prove several upper bounds for $\mathsf D(G)$ when $G$ is a $p$--group. Chief among these, that $\mathsf D(G)\leq \frac{p^2+2p-2}{p^3}|G|$ holds for a non-abelian $p$--group $G$. In Section \ref{sec-pq-group}, we tackle the main group of difficulty in this paper---the non-abelian group of order $pq$---and determine the exact value of the large Davenport constant for such groups (the small Davenport constant of these groups was previously computed by Bass \cite{Ba07b}). The methods used in Section \ref{sec-pq-group} will then be put to further use in Section \ref{sec-near-dihedral} to determine the small Davenport constant of another problematic group: $G=\la\alpha,\,\tau:\; \alpha^q=1,\quad \tau^{4}=1,\quad \alpha\tau=\tau\alpha^r\ra$, where $q$ is an odd prime and $r^2\equiv -1\mod q$. Finally, in  Section \ref{sec-genbounds}, making full use of the previous results as well as the main result from \cite{DavI}, we prove two general upper bounds for $\mathsf D(G)$ when $G$ is non-cyclic. First, that $\mathsf D(G)\leq \frac2p |G|$, where $p$ is the smallest prime divisor of $|G|$; and second, that $\mathsf D(G)\leq \frac34|G|$, provided that $G$ is also not dihedral of order $2n$ with $n$ odd (it is known that $\mathsf D(G)=|G|$ for such groups \cite{DavI}). The latter  mirrors a similar upper bound for the Noether constant from Invariant Theory \cite{Ne07b}.


\section{Notation and Preliminaries}\label{sec-notation}

As mentioned already, we use the notational conventions described in detail in \cite{DavI} as well as the results found there. However, we have need of a small amount of additional notation and results. First, if $G$ is a group, $X\subseteq G$ is a subset, and $S\in \Fc(G)$ is a sequence over $G$, then
$$\vp_X(S)=\Summ{x\in X}\vp_x(S)$$ is the number of terms of $S$ from $X$.
Second, we need the natural extension of the subsequence sum and product notation defined in \cite{DavI}: $$\Sigma_{\leq n}(S)=\bigcup_{h\in [1,n]}\Sigma_h(S)\quad\und\quad \Pi_n(S)=\bigcup_{h\in [1,n]}\Pi_h(S),$$ where $S\in \Fc(G)$ is a sequence over a group $G$.
We also have need of the characterization of equality in the Cauchy-Davenport Theorem \cite[Theorem 4.1]{DavI} \cite[Theorem 2.2]{natboook} \cite[Theorem 6.2]{Gr13a}, which was done by Vosper \cite[Theorem 2.4]{natboook} \cite[Theorem 8.1]{Gr13a}.

\begin{theorem}[Vosper's Theorem]
Let $G\cong C_p$ with $p$ prime and let $A,\,B\subseteq G$ be nonempty subsets with $|A|,\,|B|\geq 2$. If
$$|A+B|<\min\{p-1,\,|A|+|B|\},$$ then $A$ and $B$ are arithmetic progressions of common difference.\end{theorem}

Many of our arguments rely upon the use of a subgroup $H\leq G$ acting upon the finite group $G$  by conjugation (see \cite[Chapter 1]{Robinson}). We use fairly standard notation for this. For a subset $A\subseteq G$ and $x\in G$, we let $$x^A=\{a^{-1}xa:\; a\in A\}.$$ More generally, if $A,\,B\subseteq G$, then $$A^B=\{b^{-1}ab:\;a\in A,\,b\in B\}.$$ Thus $a^H$  is the $H$-orbit of $a$ under the action of conjugation by elements from $H\leq G$, which has size \be\label{orbit-size}|a^H|=|H|/|\mathsf C_G(a)\cap H|.\ee

For a finite group $G$, we let $\eta(G)$ denote the minimal integer such that every sequence $S\in \Fc(G)$ with length $|S|\geq \eta(G)$ has a nontrivial product-one subsequence of length at most $\max\{\ord(g):\; g\in G\}$. When $G=C_n^2$ with $n\geq 2$, we have $\max\{\ord(g):\; g\in G\}=\exp(G)=n$,  and both the constants $\eta(G)$ and $\mathsf D(G)$ are  known
\cite[Theorem 5.8.3]{Ge-HK06a}: \be \label{eta-p2} \eta(C_n^2)=3n-2\quad \und\quad
\mathsf d(C_n^2)+1=\mathsf D(C_n^2)=2n-1.\ee

Finally, as noted in \cite[Section 2]{DavI}, given any sequence $S\in \Fc(G)$, we have $\pi(G)$ contained in a $G'$-coset, where $G'=[G,G]\leq G$ is the commutator subgroup. Thus $\pi(S)=Ag$ for some $A\subseteq G'$ and $g\in G$. Consequently, if we have sequences $S_1,\ldots,S_\ell\in\Fc(G)$, then, for each $i\in [1,\ell]$, we have  $$\pi(S_i)=A_ig_i\quad\mbox{ for some }\quad A_i\subseteq G'\quad\und\quad g_i\in G.$$ Furthermore, since $G'\unlhd G$ is a normal subgroup, and thus invariant under conjugation automorphisms, it follows, for each $j\in [1,\ell]$, that  $g_1\ldots g_{j-1}A_j=A'_jg_1\ldots g_{j-1} $ for some $A'_j\subseteq G'$ with $|A'_j|=|A_j|=|\pi(S_j)|$. Specifically, $A'_j=A_j^{(g_1\ldots g_{j-1})^{-1}}$. Thus $$\pi(S_1)\pi(S_2)\ldots \pi(S_\ell)=(A_1g_1)(A_2g_2)\ldots (A_\ell g_\ell)=A'_1A'_2\ldots A'_\ell g,$$ where $g=g_1\ldots g_\ell$.
In particular, if $G'\cong C_q$ with $q$ prime, then theorems estimating the cardinality of a sumset in $C_q$, such as the Cauchy-Davenport Theorem, can be applied to estimate the cardinality of the product-set $\pi(S_1)\ldots \pi(S_\ell)$. We will frequently do this without further reference to the intermediary sets $A'_i\subseteq G'$.



\section{Upper Bounds Involving $\mathsf d(G)$ and $|G'|$}\label{sec-G'}

As noted in \eqref{basic-bound-d-D}, we have $\mathsf d(G)+1\leq \mathsf D(G)$ with equality if $G$ is abelian. In this section, we show that the inequality $\mathsf d(G)+1\leq \mathsf D(G)$ cannot be far from equality. Indeed, the closer $G$ is to being abelian (as measured by the commutator $G'=[G,G]$), the closer  $\mathsf D(G)$ is bounded to $\mathsf d(G)+1$. The main result of the section is the following.

\begin{theorem} \label{lem-commutator}
Let $G$ be a finite group.
Then  \be\nn\mathsf D(G)\leq \mathsf d(G)+2|G'|-1,\ee where $G'=[G,G]\leq G$ is the commutator subgroup of $G$,
 with equality if and only if $G$ is abelian.\end{theorem}

The proof of Theorem \ref{lem-commutator} will be given at the end of the section.
Before continuing, we make the following easy observation.

\begin{lemma}\label{lem-center}
Let $G$ be a group. If $x,\,y\in G$ are elements such that $xy\neq yx$, then $xy\notin \mathsf Z(G)$.
\end{lemma}

\begin{proof}
Assume by contradiction that $xy\in \mathsf Z(G)$. Then $xyx^{-1}y^{-1}=x^{-1}xyy^{-1}=1$, which implies $xy=yx$, contrary to hypothesis.
\end{proof}

We continue with
 an extremely important technical lemma embodying a simple algorithm at the heart of many of the proofs. We need several variations on the algorithm, which accounts for the rather weighty and technical formulation of Lemma \ref{lem-algorithm}.

\begin{lemma}\label{lem-algorithm} Let $G$ be a non-abelian, finite group, let $S^*\in \Fc^*(G)$ be an ordered sequence, let $H\leq G$ be an abelian subgroup, let $$\omega\geq 1,\quad
\omega_H\in \Z, \quad \und\quad \omega_0\in \{0\}\cup [2,|S^*|]\;\mbox{ with }\; \omega_0\leq \omega,$$ and suppose  that $|\pi(S_0)|\geq |S_0|=\omega_0$ and  $\pi(S_0)\cap \big(G\setminus \mathsf Z(G)\big)\neq \emptyset$ (if $\omega_0>0$), where $S_0=[S^*(1,\omega_0)]$, and that there are at least $\omega_H$ terms of $S_0^{[-1]}\bdot S$ from $H$.

Then there exists an ordered sequence ${S'}^*\in\Fc(G)$ with
\be\label{alg-conc-1}[{S'}^*]=[S^*]\quad\und\quad\pi({S'}^*)\in \pi(S^*)^G,\ee  having a factorization \be\label{TheTi}{S'}^*=T^*_1\bdot\ldots \bdot T^*_{r-1}\bdot T^*_r\bdot R^*,\ee where $T^*_1,\ldots,T^*_r,\,R^*\in \Fc^*(G)$ and  $r\geq 0$, such that, letting $R=[R^*]$ and $T_i=[T^*_i]$ for $i\in [1,r]$, we have  $S_0\mid T_1$ (if $\omega_0>0$), \be\label{alg-conc-2}\pi(T_i)\cap \big(G\setminus \mathsf Z(G)\big)\neq \emptyset \quad\und\quad |\pi(T_i)|\geq |T_i|\geq 2\quad\mbox{ for $i\in [1,r]$}, \quad \pi(T_i)^G=\pi(T_i)\;\mbox{ for $i\in [1,r-1]$},\ee and either
\begin{itemize}
\item[(i)] $\Sum{i=1}{r}|T_i|\leq w-1$ \ and \ $\la\supp(R)\ra<G$ \ is a proper subgroup, or
\item[(ii)] $w\leq \Sum{i=1}{r}|T_i|\leq w+1$, with the upper bound only possible if $|T_r|=2$ and $\Sum{i=1}{r-1}|T_i|=\omega-1$, and there are at least $\omega_H$ terms of $R$ from $H$, or
\item[(iii)] $\Sum{i=1}{r}|T_i|\leq w-1$ and there are precisely $\omega_H$ terms of $R$ from $H$.
\end{itemize}
\end{lemma}

\begin{proof}Let $S=[S^*]$.
First observe that $$\pi\Big(S^*(2,|S|)\bdot S^*(1)\Big)=S^*(1)^{-1}\pi(S^*)S^*(1).$$ Thus cyclically shifting the terms of $S^*$ results in an ordered sequence ${S'}^*$ with $[{S'}^*]=[S^*]=S$ and product in $\pi(S^*)^G$, as required by \eqref{alg-conc-1}.

Let ${S'}^*\in \Fc^*(G)$ be an ordered sequence with a factorization (and all notation) given by \eqref{TheTi}, satisfying all parts of the lemma apart from (possibly) conclusions (i)--(iii), with at least $\omega_H$ terms of $R$ from $H$, with $|T_r|=2$ if $\Sum{i=1}{r}|T_i|=w+1$, and subject to all this, with $\Sum{i=1}{r}|T_i|\leq w+1$ maximal. We begin by showing that such an ordered sequence ${S'}^*$ exists.

If $\omega_0=0$, then all conclusions of the lemma apart from (i)--(iii) hold taking $R^*=S^*={S'}^*$ and $r=0$; moreover, we know $S^*=R^*$ contains at least $\omega_H$ terms from $H$ by hypothesis, and clearly $\Sum{i=1}{r}|T_i|=0<\omega$. Thus the ${S'}^*$ described above exists in the case $\omega_0=0$. On the other hand, if $\omega_0\geq 2$ (note $\omega_0=1$ is not allowed by our hypotheses),
then all conclusions of the lemma apart from (i)--(iii) hold taking $S^*={S'}^*$, \ $R^*=S^*(\omega_0+1,|S^*|)$, \ $r=1$ and $T^*_1=S^*(1,\omega_0)$ (as follows from the hypotheses); moreover, we know that $|T_1|=\omega_0\leq \omega<\omega+1$ and that $R^*=S^*(\omega_0+1,|S^*|)$ contains at least $\omega_H$ terms from $H$ by hypothesis. Thus the ${S'}^*$ described above exists in the case $\omega_0\geq 2$ as well.

If $\Sum{i=1}{r}|T_i|\geq \omega$, then (ii) holds and the proof is complete. Therefore we can assume \be\label{sheep1}\Sum{i=1}{r}|T_i|\leq  \omega-1.\ee Hence, if there are precisely $\omega_H$ terms of $R$ from $H$, then (iii) holds and the proof is again complete. Therefore, since there are assumed to be at least $\omega_H$ terms of $R$ from $H$, it follows that this estimate must be strict: \be\vp_H(R)\geq \omega_H+1\label{sheep2}.\ee If $\la\supp(R)\ra<G$ is a proper subgroup, then  (i) holds, completing the proof once more. Therefore we can assume \be\label{sheep3}\la \supp(R)\ra=G.\ee We now aim to show that \eqref{sheep1}--\eqref{sheep3} allow us to contradict the maximality of $\Sum{i=1}{r}|T_i|$ for ${S'}^*$. We proceed in two cases.

\subsection*{Case 1:} $r\geq 1$ and $\pi(T_r)^G\neq \pi(T_r)$.

If $\pi(T_r)^{\supp(R)}=\pi(T_r)$, then it is easily shown that $\pi(T_r)^{\la\supp(R)\ra}=\pi(T_r)$. But since \eqref{sheep3} gives $\la \supp(R)\ra=G$, this would mean  $\pi(T_r)^G=\pi(T_r)^{\la\supp(R)\ra}=\pi(T_r)$, contrary to case hypothesis.
Therefore there must be some $g\in \supp(R)$ such that $g\pi(T_r)\neq \pi(T_r)g$. Let $x\in [1,|R|]$ be minimal such that $R^*(x)\pi(T_r)\neq \pi(T_r)R^*(x)$.

By the minimality of $x$, we have $R^*(y)\pi(T_r)=\pi(T_r)R^*(y)$ for every $y\in [1,x-1]$.
Since $\pi(T_j)^G=\pi(T_j)$ for $j\in [1,r-1]$ (in view of   \eqref{TheTi} holding for ${S'}^*$), we also have $R^*(y)\pi(T_j)=\pi(T_j)R^*(y)$ for every $y\in [1,x-1]$ and $j\in [1,r-1]$.
Thus, for each $j\in [1,r]$, there exists ordering ${T'}^*_j$ of $T_j$ such that $$\pi\Big(R^*(1,x-1)\Big)\pi({T'}^*_j)=\pi(T^*_j)\pi\Big(R^*(1,x-1)\Big).$$
Hence $$\pi\Big(R^*(1,x-1)\Big)\pi({T'}^*_1)\ldots\pi({T'}^*_r)=\pi(T^*_1)\ldots
\pi(T^*_r)\pi\Big(R^*(1,x-1)\Big).$$ In other words, allowing re-ordering of the terms of the $T_i$, we can commute the terms from $R^*(1,x-1)$ past the $T_i$ while preserving that the resulting ordered sequence still has the same product. Then, as mentioned at the beginning of the proof, we can cyclically shift the terms $R^*(1,x-1)$ until the sequence ${T'}^*_1$ is once again the start of the resulting sequence $${S''}^*:={T'}^*_1\bdot \ldots \bdot {T'}^*_r \bdot R^*(x,|R|)\bdot R^*(1,x-1),$$ and this will preserve that $\pi({S''}^*)\in \pi({S'}^*)^G=\pi(S^*)^G$ (note that the hypothesis  $\pi({S'}^*)\in \pi(S^*)^G$ is equivalent to $\pi({S'}^*)^G=\pi({S}^*)^G$).
Moreover, this does not affect any of the defining properties of the $T_i$, which means that (by replacing ${S'}^*$ by ${S''}^*$, the $T^*_i$ by the ${T'}^*_i$, and $R^*$ by $R^*(x,|R|)\bdot R^*(1,x-1)$), we can w.l.o.g. assume $x=1$.

Observe that $$\pi(T_r)R^*(1)\cup R^*(1)\pi(T_r)\subseteq \pi\Big(T_r\bdot R^*(1)\Big).$$ Moreover, we have $\pi(T_r)R^*(1)\neq R^*(1)\pi(T_r)$ in view of $x=1$ and the definition of $x$. Thus $$|\pi\Big(T_r\bdot R^*(1)\Big)|\geq |\pi(T_r)|+1.$$ If $gR^*(1)=R^*(1)g$ for every $g\in \pi(T_r)$, then $R^*(1)\pi(T_r)=\pi(T_r)R^*(1)$ would follow, contrary to the definition of $x=1$. Therefore there must be some $g\in \pi(T_r)$ with $R^*(1)g\neq gR^*(1)$, in which case  Lemma \ref{lem-center} ensures that the element $gR^*(1)\in \pi\Big(T_r\bdot R^*(1)\Big)$ is from $G\setminus \mathsf Z(G)$. In view of \eqref{sheep2}, we see that ${R'}^*:=R^*(2,|R|)$ contains at least $\omega_H$ terms from $H$, and \eqref{sheep1} ensures that $\Sum{i=1}{r}|T'_i|\leq \omega<\omega+1$, where ${T'_i}^*:={T_i}^*$ for $i\in [1,r-1]$ and ${T'_r}^*:={T_r}^*\bdot R^*(1)$. But now the maximality of $\Sum{i=1}{r}|T_i|$ for ${S'}^*$ is contradicted by the factorization ${S'}^*={T'}^*_1\bdot\ldots\bdot {T'}^*_r\bdot {R'}^*$, completing Case 1.

\subsection*{Case 2:} $r=0$ or $\pi(T_r)^G= \pi(T_r)$.

If $gh=hg$ for all $g,\,h\in \supp(R)$, then $\la \supp(R)\ra$ must be abelian. Hence, since $G$ is non-abelian by hypothesis, $\la \supp(R)\ra$ is  a proper subgroup of $G$, contrary to \eqref{sheep3}. Therefore, there must be $g_0,\,h_0\in \supp(R)$ with $g_0h_0\neq h_0g_0$.
Swapping the order of adjacent terms of $R^*$ that commute with each other  preserves all assumptions from the definition of ${S'}^*$.
Consequently, performing such swaps, we can either  arrange that $R^*$  has the non-commuting terms $g_0$ and $h_0$ adjacent to each other or else has $g_0$ adjacent to another term $h'_0$ that also does not commute with $g_0$. Either way, we may assume there are consecutive terms in $R^*$ that do not commute, say $R^*(x)R^*(x+1)\neq R^*(x+1)R^*(x)$ with $x\in [1,|R|-1]$.

By \eqref{TheTi} and case hypothesis, we have $\pi(T_j)g=g\pi(T_j)$ for all $g\in G$ and $j\in [1,r]$. Thus, as we argued in Case 1, we can commute the terms $R^*(1,x-1)$ past the $T^*_i$, re-ordering each $T_i$ appropriately, and then cyclically shift the  terms $R^*(1,x-1)$ to thereby w.l.o.g. assume $x=1$.

Let $T^*_{r+1}:=R^*(1,2)$, \ $T_{r+1}=[T^*_{r+1}]$ and ${R'}^*:=R^*(3,|R|)$. Since $R^*(1)R^*(2)\neq R^*(2)R^*(1)$ (in view of the definition of $x=1$), we have $|\pi(T_{r+1})|\geq 2=|T_{r+1}|$ while Lemma \ref{lem-center} ensures that $\pi(T_{r+1})\cap (G\setminus \mathsf Z(G))\neq \emptyset$. In view of the case hypothesis, we have $\pi(T_r)^G= \pi(T_r)$, while $\pi(T_j)^G= \pi(T_j)$ holds for $j\in [1,r-1]$ from the hypotheses in the definition of ${S'}^*$.
Since $H$ is abelian and the terms $R^*(1)$ and $R^*(2)$ do not commute with each other, it follows that at most one term from $R^*(1,2)$ is from $H$, whence \eqref{sheep2} ensures that ${R'}^*$ contains at least $\omega_H$ terms from $H$. By its definition, we have $|T_{r+1}|=2$, and \eqref{sheep1} gives $\Sum{i=1}{r+1}|T_i|\leq \omega-1+2=\omega+1$. But now the maximality of $\Sum{i=1}{r}|T_i|$ for ${S'}^*$ is contradicted by the factorization ${S'}^*={T}^*_1\bdot\ldots\bdot {T}^*_r\bdot T^*_{r+1}\bdot {R'}^*$, completing Case 2 and the proof.
\end{proof}

Next, we give a simple application of Lemma \ref{lem-algorithm}. Note that the $p$ defined in Corollary \ref{lem-commutator-I} is always at least as big as the smallest prime divisor of $|G|$.

\begin{corollary} \label{lem-commutator-I}
Let $G$ be a finite, non-abelian group, let $G'=[G,G]\leq G$ be its commutator subgroup, and let
$$p=\min\{|G|/|\mathsf C_G(x)|:\; x\in G\setminus \mathsf Z(G)\}.$$
Suppose $G'$ is cyclic of prime order. Then \ber\nn\mathsf D(G)&\leq& \max\left(\left\{\mathsf d(G)+|G'|+\left\lfloor\frac{|G'|-2}{p-1}\right\rfloor\right\}\cup\left\{\mathsf D(H)+|G'|+\left\lfloor\frac{|G'|-2}{p-1}\right\rfloor-2:\; H<G\mbox{ proper}\right\}\right).\eer

In particular, if we also know that all proper subgroups $H<G$ are abelian, then $$\mathsf D(G)\leq \mathsf d(G)+|G'|+\left\lfloor\frac{|G'|-2}{p-1}\right\rfloor.$$
\end{corollary}



\begin{proof}
Since $G$ is non-abelian, $G'\leq G$ is nontrivial and  $\mathsf Z(G)<G$ is proper.
Note that the ``in particular'' statement of the corollary follows from the main part in view of the inequality $\mathsf D(H)=\dd(H)+1\leq \dd(G)+1$ holding for any abelian subgroup $H\leq G$ (care of \cite[Lemma 2.4.3]{DavI}). In view of \eqref{orbit-size}, we see that $p\geq 2$ is the minimal size of an orbit of an element $g\in G\setminus \mathsf Z(G)$. Assume by contradiction that we have an atom $S\in \mathcal A(G)$ with \be\label{S-big}|S|\geq \max\left(\left\{\mathsf d(G)+|G'|+\left\lfloor\frac{|G'|-2}{p-1}\right\rfloor+1\right\}\cup\left\{\mathsf D(H)+|G'|+\left\lfloor\frac{|G'|-2}{p-1}\right\rfloor-1:\; H<G\mbox{ proper}\right\}\right).\ee
Since $S\in \mathcal A(G)$, there is an ordering $S^*\in \Fc^*(G)$ with $[S^*]=S$ and $\pi(S^*)=1$.

Apply Lemma \ref{lem-algorithm} to $S^*$ taking $H$ trivial, $\omega=|G'|+\left\lfloor\frac{|G'|-2}{p-1}\right\rfloor$, \ $\omega_H=-1$, and $\omega_0=0$ and let $${S'}^*=T^*_1\bdot \ldots\bdot T^*_r\bdot R^*$$ be the resulting factorization, where $T^*_1,\ldots,T^*_r,\,R^*\in \Fc^*(G)$, \  $[R^*]=R$ and $[T^*_i]=T_i$ for $i\in [1,r]$.
Observe that \be\label{omega-big} \omega\geq \frac{p|G'|-p}{p-1}.\ee
 Since $\omega_H$ is negative, Lemma \ref{lem-algorithm}(iii) cannot hold.
This gives us two cases.

\subsection*{Case 1:} Lemma \ref{lem-algorithm}(i) holds.

Since $\pi({S'}^*)\in \pi(S^*)^G=1^G=\{1\}$ (from \eqref{alg-conc-1}), we see that ${S'}^*$ is a product-one ordered sequence. In view of Lemma \ref{lem-algorithm}(i), we have $\la \supp(R)\ra:=H<G$ being a proper subgroup. In view of  Lemma \ref{lem-algorithm}(i) and \eqref{S-big}, we also know $$|R|\geq |S|-w+1\geq \left(\mathsf D(H)+|G'|+\left\lfloor\frac{|G'|-2}{p-1}\right\rfloor-1\right)-
|G'|-\left\lfloor\frac{|G'|-2}{p-1}\right\rfloor+1= \mathsf D(H).$$
Thus we can apply \cite[Lemma 2.5]{DavI} to $R$ to find a nontrivial, product-one subsequence $T\mid R$ with $|T|\leq \mathsf D(H)<|S|$ and $\pi(R^*)\in\pi (T^{[-1]}\bdot R)$. Hence $$1=\pi({S'}^*)=\pi(T^*_1)\ldots \pi(T^*_r)\pi(R^*)\in \pi\Big(T_1\bdot\ldots\bdot T_r\bdot (T^{[-1]}\bdot R)\Big)=\pi(T^{[-1]}\bdot S),$$ which means $S=T\bdot (T^{[-1]}\bdot S)$ is a nontrivial factorization, contradicting that $S\in \mathcal A(G)$ is an atom.

\subsection*{Case 2:} Lemma \ref{lem-algorithm}(ii) holds.

Since $\pi(T_i)\cap \Big( G\setminus \mathsf Z(G)\Big)\neq \emptyset$ and $\pi(T_i)^G=\pi(T_i)$ for $i\in [1,r-1]$ (from \eqref{alg-conc-2}), it follows, in view of  the description of $p$ given at the beginning of the proof, that \be\label{maddhat}|\pi(T_r)|\geq 2\quad\und\quad |\pi(T_i)|\geq p\quad\mbox{ for all $i\in [1,r-1]$},\ee  where the first inequality follows directly from \eqref{alg-conc-2}.

\medskip

Suppose we can find a subsequence $T\mid S$ such that $\pi(T)$ is a full $G'$-coset and $|T|\leq \omega$. Then, in view of \eqref{S-big}, we have
$$|T^{[-1]}\bdot S|\geq |S|-\omega\geq \mathsf d(G)+1.$$ As a result, the definition of $\mathsf d(G)$  guarantees that there is a nontrivial, product-one subsequence $V_1\mid T^{[-1]}\bdot S$. Thus $S=V_1\bdot V_2$ with $T\mid V_2$, where $V_2=V_1^{[-1]}\bdot S$. Note that $V_2={V_1}^{[-1]}\bdot S$ is nontrivial since it contains the subsequence $T$ which must be nontrivial in view of $\pi(T)$ being a full $G'$-coset with $G'$ nontrivial. By
\cite[Lemma 2.2]{DavI}, we have \be\label{redred}\pi(V_2)\subseteq G'.\ee Since $T\mid V_2$, and since $\pi(T)$ is a full $G'$-coset, it likewise follows that $\pi(V_2)$ is also a full $G'$-coset, meaning the inclusion in \eqref{redred} is an equality: $1\in G'=\pi(V_2)$. Consequently, $S=V_1\bdot V_2$ is   a factorization of $S$ into two nontrivial, product-one subsequences, contradicting that $S\in \mathcal A(G)$ is an atom. So we instead assume that \be \label{no-good-T}\mbox{there does not exist a subsequence }\;T\mid S\;\mbox{ with }\quad |\pi(T)|=|G'|\quad\und\quad |T|\leq \omega.\ee

\medskip


Let $W=T_1\bdot\ldots\bdot T_{r-1}$. Then $$|W|= \Sum{i=1}{r}|T_i|-|T_r|\leq \omega+1-2=\omega-1,$$ with the inequality above following from those given in Lemma \ref{lem-algorithm}(ii) and \eqref{alg-conc-2}. Thus \eqref{no-good-T} ensures that \be\label{wisk}|\pi(W)|\leq |G'|-1.\ee
Observe that  \be\label{yaith}\pi(T_1)\ldots\pi(T_{r-1})\subseteq \pi(T_1\bdot\ldots\bdot T_{r-1})
= \pi(W).\ee
Thus, since $G'$ is cyclic of prime order by hypothesis, using \eqref{wisk}, \eqref{yaith}, the Cauchy-Davenport Theorem,  and \eqref{maddhat}, we obtain  \be|G'|-1\geq |\pi(W)|\geq \Sum{i=1}{r-1}|\pi(T_i)|-r+2\geq (r-1)p-r+2.\label{notmadhat}\ee  Rearranging this inequality gives
\be r\leq \frac{|G'|+p-3}{p-1}.\label{nextplease}\ee
In view of  Lemma \ref{lem-algorithm}(ii) holding by case hypothesis, we have $\omega\leq \Sum{i=1}{r}|T_i|\leq \omega+1$.

\medskip

Suppose $\Sum{i=1}{r}|T_i|= \omega+1$. In this case, Lemma \ref{lem-algorithm}(ii) further tells us that $|T_r|=2$ and  $\Sum{i=1}{r-1}|T_i|=\omega-1$; and from \eqref{alg-conc-2}, we have $|\pi(T_i)|\geq |T_i|$ for all $i$. Thus \eqref{notmadhat} and \eqref{nextplease} yield  \be\nn|G'|-1\geq \Sum{i=1}{r-1}|\pi(T_i)|-r+2\geq \Sum{i=1}{r-1}|T_i|-r+2=\omega+1-r\geq \omega+1-\frac{|G'|+p-3}{p-1}.\ee Consequently, $\omega\leq \frac{p|G'|-p-1}{p-1}$, contradicting \eqref{omega-big}. So we instead conclude that \be\label{batflys} \Sum{i=1}{r}|T_i|= \omega.\ee

In view of \eqref{batflys}, we have $|T_1\bdot\ldots\bdot T_r|\leq \omega$. But now we obtain a string of inequalities as follows: the first inequality follows from  \eqref{no-good-T}, the second is clear, the third from an application of the Cauchy-Davenport Theorem as argued for \eqref{notmadhat}, the fourth in view of \eqref{alg-conc-2},  the equality from \eqref{batflys}, and the final inequality from \eqref{nextplease}. \ber\nn |G'|-1&\geq& |\pi(T_1\bdot \ldots\bdot T_r)|\geq |\prod_{i=1}^{r}\pi(T_i)|\geq \Sum{i=1}{r}|\pi(T_i)|-r+1\\&\geq& \Sum{i=1}{r}|T_i|-r+1=\omega+1-r\geq \omega+1-\frac{|G'|+p-3}{p-1}\nn\eer Rearranging the above inequality gives $\omega\leq \frac{p|G'|-p-1}{p-1},$ contrary to \eqref{omega-big}, which completes the proof.\end{proof}

We conclude the section with the proof of Theorem \ref{lem-commutator}.

\begin{proof}[Proof of Theorem \ref{lem-commutator}] We have $\mathsf D(G)=\mathsf d(G)+1$ for any abelian group $G$ (care of \cite[Lemma 2.4.3]{DavI}). Thus it suffices to show \be\nn\mathsf D(G)\leq \mathsf d(G)+2|G'|-2\ee for a finite, non-abelian group $G$. Since $G$ is non-abelian, $G'$ is nontrivial.

Let $U\in \mathcal A(G)$ be an atom with $|U|=\mathsf D(G)$.
As in the proof of \eqref{no-good-T} in Corollary   \ref{lem-commutator-I},
 may  assume \be\label{NoexistX}\mbox{ there is no subsequence \ $T\mid U$  \ with  \ $|T^{[-1]}\bdot U|\geq \mathsf d(G)+1$ \ and  \ $\pi(T)$ \ a full \ $G'$-coset}\ee and, by  way of contradiction, that
 \be\label{Ulonger} |U|\geq \mathsf d(G)+2|G'|-1.\ee

\bigskip


Let $\ell\in [2,2|G'|-2]$ be the maximal integer such that there exists an ordered sequence $U^*\in \mathcal F^*(G)$ with \be\label{cub0}[U^*]=U\quad\und\quad \pi(U^*)=1\ee having a factorization $U^*=T^*\bdot R^*$, where  $R^*,\,T^*\in \Fc^*(G)$, \ $R:=[R^*]$ and $T:=[T^*]$, such that
\ber\label{cubi}&&|T|=\ell
\quad\und\quad |\pi(T)|\geq \frac12 |T|+1.\eer
To see that $\ell\geq 2$ exists, we argue as follows.  If $\la\supp(U)\ra:=H$ were abelian, then $H<G$ follows since   $G$ is non-abelian, and then Lemma \cite[Lemma 2.4.3]{DavI} gives $|U|\leq \mathsf D(H)=\mathsf d(H)+1\leq \mathsf d(G)+1$, contrary to \eqref{Ulonger}. Therefore we can assume there are terms $g_0,\,h_0\in \supp(U)$ that do not commute: $g_0h_0\neq h_0g_0$. But now, arguing as from the beginning of Case 2 in Lemma \ref{lem-algorithm} allows us to w.l.o.g. assume the first two terms of $U^*$ do not commute, in which case it is clear that $\ell\geq 2$ exists.
Also, if $\ell$ were odd, then taking $T^*\bdot R^*(1)$ in place of $T^*$ would contradict the maximality of $\ell$, which means that $\ell$ must be even. Finally, if $\ell=2|G'|-2$, then the sequence $T$ will contradict \eqref{NoexistX} in view of \eqref{Ulonger}. Thus, since $\ell$ is even, we must have \be\label{max-contradict}2\leq \ell\leq 2|G'|-4.\ee

In view of $\ell\leq 2|G'|-4$ and \eqref{Ulonger}, we have $|R|=|U|-\ell\geq \dd(G)+3$. Thus the definition of $\dd(G)$ guarantees that $R$ has  a nontrivial, product-one subsequence. Consequently, we can reorder the terms of $R^*$ so that the resulting ordered sequence ${R'}^*$ has a nontrivial, product-one consecutive subsequence. Of course, we may have  $\pi({R'}^*)\neq \pi(R^*)$.
It is well-known that the symmetric group on $|R|$ elements can be generated by the cycles $(1,2)$  and $(1,2,\ldots,|R|)$. But this means that there is a chain of ordered sequences  $$R^*_0,R^*_1,\ldots,R^*_n\in \Fc^*(G)$$ such that \begin{align} \nn & R^*_0=R^*,\quad R^*_n={R'}^*, \quad [R^*_i]=R\quad\mbox{ for all $i\in [1,n]$,}\quad\mbox{ and either}\\\label{induct-cong} & R^*_{i+1}=R^*_{i}(2,|R|)\bdot R^*_{i}(1)\quad\mbox{ or }\quad R^*_{i+1}=R^*_{i}(2)\bdot R^*_{i}(1)\bdot R^*_{i}(3,|R|)\quad\mbox{ for each $i\in [0,n-1]$}.\end{align}

Since $1=\pi(U^*)=\pi(T^*\bdot R^*)=\pi(T^*\bdot R^*_0)$, we have $$\pi(R^*_0)^{-1}=\pi(T^*)\in \pi(T).$$ If $\pi(R^*_n)^{-1}\in \pi(T)$, then we could order the terms of $T$, yielding some $T^*_n\in \Fc^*(G)$ with $[T^*_n]=T$, such that $\pi(T^*_n\bdot R^*_n)=1$. But then, since $R^*_n={R'}^*$ contains a nontrivial, consecutive product-one subsequence, say $R^*_n(I)$ with $I\subseteq [1,|R^*_n|]$ a nonempty interval, it would follow that $$U=T\bdot R=\left[T^*_n\bdot R^*_n\Big([1,|R|]\setminus I\Big)\right]\bdot \left[R^*_n(I)\right]$$ was a factorization of $U$ into $2$ nontrivial product-one subsequences---note $[T^*_n\bdot R^*_n([1,|R|]\setminus I)]$ is also nontrivial since it contains $[T^*_n]=T$ and $|T|=\ell\geq 2$---contradicting that $U\in \mathcal A(G)$ is an atom. Therefore we can instead assume that $$\pi(R^*_n)^{-1}\notin \pi(T).$$

As a result, let $s+1\in [1,n]$ be the minimal integer such that
\be\label{sdef}\pi(R^*_{s+1})^{-1}\notin \pi(T).\ee
In view of the minimality of $s+1\in [1,n]$, it follows that $\pi(R^*_{s})^{-1}\in \pi(T)$, which means that we can order  the terms of $T$, yielding some $T^*_{s}\in \Fc^*(G)$ with $[T^*_{s}]=T$, such that $$\pi(T^*_{s}\bdot R^*_{s})=1.$$
In view of \eqref{induct-cong}, there are $2$ possibilities for how $R^*_{s+1}$ was obtained from $R^*_s$.

\medskip

Suppose first that $R^*_{s+1}=R^*_s(2,|R|)\bdot R^*_s(1)$. If  $$\pi(T) R_s^*(1)=R_s^*(1) \pi(T),$$ then the terms of $T$ can be ordered, yielding some $T^*_{s+1}\in \Fc^*(G)$ with $[T^*_{s+1}]=T$, such that $$R^*_s(1)\bdot T^*_{s+1}\bdot R^*_s(2,|R|)\in \Fc^*(G)$$  has product $\pi(T^*_s\bdot R^*_s)=1$. But then Lemma \cite[Lemma 2.3]{DavI} implies that $$T^*_{s+1}\bdot R^*_s(2,|R|)\bdot R^*_s(1)=T^*_{s+1}\bdot R^*_{s+1}$$ also has product one, in which case $\pi(R^*_{s+1})^{-1}\in \pi([T^*_{s+1}])=\pi(T)$, contradicting \eqref{sdef}. Therefore, we instead conclude that $\pi(T) R^*_s(1)\neq R^*_s(1)\pi(T).$  Consequently, since $$\pi(T)R^*_s(1)\cup R^*_s(1)\pi(T)\subseteq
\pi\Big(T\bdot R^*_s(1)\Big),$$ it follows in view of \eqref{cubi} that \be\label{gallweyI} |\pi\Big(T\bdot R^*_s(1)\Big)|\geq |\pi(T)|+1\geq \frac12|T|+2\geq \frac12 {|T\bdot R^*_s(1)|}+1.\ee
Thus, in view of \eqref{max-contradict}, the maximality of $\ell\in [2,2|G'|-2]$ is contradicted by $T\bdot R^*_s(1)$ taking $U^*=T^*_s\bdot R^*_s$ for \eqref{cub0}. So we may instead assume that  $$R^*_{s+1}=R^*_{s}(2)\bdot R^*_{s}(1)\bdot R^*_{s}(3,|R|).$$

\medskip

The remainder of the proof is now just a variation on the previous paragraph.
If  $$\pi(T)R^*_s(1)R^*_s(2)=\pi(T)R^*_s(2)R^*_s(1),$$ then the terms of $T$ can be ordered, yielding some $T^*_{s+1}\in \Fc^*(G)$ with $[T^*_{s+1}]=T$, such that $$T^*_{s+1}\bdot R^*_{s+1}=T^*_{s+1}\bdot R^*_s(2)\bdot R^*(1)\bdot R^*_s(3,|R|)\in \Fc^*(G)$$  has product $\pi(T^*_s\bdot R^*_s)=1$, in which case $\pi(R^*_{s+1})^{-1}\in \pi([T^*_{s+1}])=\pi(T)$, contradicting \eqref{sdef}. Therefore, we instead conclude that $\pi(T)R^*_s(1)R^*_s(2)\neq \pi(T)R^*_s(2)R^*_s(1).$  Consequently, since $$\pi(T)R^*_s(1)R^*_s(2)\cup\pi(T)R^*_s(2)R^*_s(1)\subseteq
\pi\Big(T\bdot R_s^*(1)\bdot R_s^*(2)\Big),$$ it follows in view of \eqref{cubi} that \be\label{gallwey} \left|\pi\Big(T\bdot R_s^*(1)\bdot R_s^*(2)\Big)\right|\geq |\pi(T)|+1\geq \frac12|T|+2= \frac12 {|T\bdot R^*_s(1)\bdot R^*_s(2)|}+1.\ee
Thus, in view of \eqref{max-contradict}, the maximality of $\ell\in [2,2|G'|-2]$ is contradicted by $T\bdot R^*_s(1)\bdot R_s^*(2)$ taking $U^*=T^*_s\bdot R^*_s$ for \eqref{cub0}, completing the proof.
\end{proof}

\section{Upper Bounds for $p$-Groups}\label{sec-pgroup}

In this section, we give general upper bounds for $\mathsf D(G)$ when $G$ is a $p$-group. The main result of the section is the following.

\begin{theorem}\label{thm-p-group} Let $G$ be a finite $p$-group with  $p\geq 2$ prime. If $G$ is non-abelian, then \be\label{p-group-bound}\mathsf D(G)\leq \frac{p^2+2p-2}{p^3}|G|.\ee
\end{theorem}

We  begin with the following lemma, which follows by standard inductive arguments.

\begin{lemma}\label{help-lemm}
Let $G$ be a finite group and let $H\unlhd G$ be a normal subgroup with $G/H\cong C_p^2$. Then $$\mathsf d(G)\leq (\mathsf d(H)+2)p-2\leq \frac{1}{p}|G|+p-2.$$
\end{lemma}

\begin{proof} From Lemma \cite[Lemma 2.4.1]{DavI}, we have $\mathsf d(H)+1\leq \mathsf D(H)\leq |H|=\frac{1}{p^2}|G|$. Hence $(\mathsf d(H)+2)p-2\leq \frac{1}{p}|G|+p-2$, so that the second inequality for the lemma holds in general.

Let $S\in \Fc(G)$ be a sequence with length $|S|\geq (\mathsf d(H)+2)p-1$. We need to show $1\in \Pi(S)$, i.e., that $S$ has a nontrivial, product-one subsequence. By hypothesis, we have $G/H\cong C_p^2$, and  from \eqref{eta-p2}, we know $\eta(G/H)=\eta(C_p^2)=3p-2$.  Repeatedly applying the definition of $\eta(G/H)$ to $\phi_H(S)$, we can remove product-one subsequences from $\phi_H(S)$ of length at most $p$ until there are at most $3p-3$ terms of $\phi_H(S)$ left. In other words, we obtain a factorization $S=[S^*_1]\bdot\ldots \bdot [S^*_\ell]\bdot [{S'}^*]$ with  $S^*_1,\ldots,S^*_\ell,{S'}^*\in \Fc^*(G)$, \ber 1\leq |S^*_i|\leq p \quad \und \quad \pi(S^*_i)\in H \quad \mbox{for $i\in [1,\ell]$},\quad \und \quad |{S'}^*|\leq 3p-3.\eer Consequently, \be\label{warff}\ell\geq \frac{|S|-|{S'}^*|}{p}\geq \frac{(\mathsf d(H)+2)p-1-|{S'}^*|}{p}\geq \frac{(\mathsf d(H)+2)p-1-3p+3}{p}=\mathsf d(H)-1+\frac{2}{p}.\ee Hence $\ell\geq \mathsf d(H)$.

If $\ell>\mathsf d(H)$, then applying the definition of $\mathsf d(H)$ to the sequence $$[\pi(S^*_1)]\bdot\ldots\bdot [\pi(S^*_\ell)]\in \Fc(H)$$ yields a nontrivial product-one subsequence $\bulletprod{i\in I}[\pi(S^*_i)]\in \Fc(H)$, for some nonempty $I\subseteq [1,\ell]$, in which case $\bulletprod{i\in I} [S^*_i]\in \Fc(G)$ is the desired product-one subsequence of $S$. So we may assume $\ell=\mathsf d(H)$.

If $|{S'}^*|\leq 2p-2$, then the estimate in \eqref{warff} improves to $\ell\geq \mathsf d(H)+1$, contrary to what we just established. Therefore $|{S'}^*|\geq 2p-1=\mathsf d(C_p^2)+1=\mathsf d(G/H)+1$ (in view of \eqref{eta-p2}). But now we can apply the definition of $\mathsf d(G/H)+1$ to the sequence $\phi_H(S')$ to find a nontrivial subsequence $[S_{\ell+1}^*]$ of $S'$ with $\pi(S_{\ell+1}^*)\in H$, where $S_{\ell+1}^*\in \Fc^*(G)$. Applying the arguments of the previous paragraph to $[\pi(S^*_1)]\bdot \ldots\bdot [\pi(S^*_{\ell+1})]$ instead of $[\pi(S^*_1)]\bdot \ldots\bdot [\pi(S^*_{\ell})]$ now yields the desired product-one subsequence of $S$, completing the proof.
\end{proof}

Now we can prove Theorem \ref{thm-p-group}.

\begin{proof}[Proof of Theorem \ref{thm-p-group}]
Since $G$ is a finite, non-abelian group, it must possess a minimal non-abelian subgroup $H\leq G$, that is, a subgroup $H\leq G$ such that all proper subgroups $K<H$ are abelian. Assuming we knew the theorem held for minimal non-abelian $p$-groups, we could apply the result to $H$ and then invoke \cite[Theorem 3.2]{DavI}, yielding the  bound $$\mathsf D(G)\leq \mathsf D(H)|G:H|\leq \frac{p^2+2p-2}{p^3}|H||G:H|=\frac{p^2+2p-2}{p^3}|G|,$$ as desired. Therefore, we see that it suffices to prove the theorem when $G$ is a minimal non-abelian group, which we now assume.

 Miller and Moreno characterized all finite minimal non-abelian groups  back in 1903 \cite{miller-moreno}. A summary of their result for finite $p$-groups can be found in the more modern text  \cite[pp. 179]{p-group-book}, with some of the details for the $p$-group case also given in  \cite{min-nonabelian-extra-ref}. We do not need the full characterization, only the following easily derived consequences: $$G'\cong C_p\quad\und \quad G/\mathsf Z(G)\cong C_p^2,$$ where $G'=[G,G]\leq G$ is the commutator subgroup.

Since $G'\cong C_p$ and all proper subgroups of $G$ are abelian (in view of $G$ being a \emph{minimal} non-abelian group), it follows from Corollary \ref{lem-commutator-I} that
\be\label{wigwoo} \mathsf D(G)\leq \mathsf d(G)+|G'|=\mathsf d(G)+p.\ee
Since $G/\mathsf Z(G)\cong C_p^2$, Lemma \ref{help-lemm} implies $\mathsf d(G)\leq \frac{1}{p}|G|+p-2$. Combining with \eqref{wigwoo}, it follows that
\be\label{wigzoo}\mathsf D(G)\leq \frac{1}{p}|G|+2p-2=\left(\frac{1}{p}+\frac{2p-2}{|G|}\right)|G|.\ee Since $G$ is a non-abelian $p$-group, we have $|G|\geq p^3$  \cite[Theorem 1.6.15]{Robinson}, which combined with \eqref{wigzoo} yields the desired bound $$\mathsf D(G)\leq \left(\frac{1}{p}+\frac{2p-2}{p^3}\right)|G|=\frac{p^2+2p-2}{p^3}|G|,$$ completing the proof.
\end{proof}

We remark that the constant $\frac{p^2+2p-2}{p^3}$ from Theorem \ref{thm-p-group} is close to optimal. The group $$M_{p^n}=\la \alpha,\,\tau:\; \alpha^{p^{n-1}}=1,\,\tau^p=1,\,\alpha\tau=\tau \alpha^{1+p^{n-2}}\ra$$ is a well-known minimal non-abelian group of order $p^n$ for $n\geq 3$, and considering the sequence $$ \tau^{p-1}\alpha\bdot \alpha^{[p-1]}\bdot \tau\alpha^{1-p}\bdot \alpha^{[p^{n-1}-1]}\in \Fc(M_{p^n})$$ shows that $\mathsf D(M_{p^n})\geq p^{n-1}+p$. When $n=3$, this gives $\mathsf D(M_{p^3})\geq p^{2}+p=\frac{p^2+p}{p^3}|M_{p^3}|$, showing that the constant $\frac{p^2+2p-2}{p^3}$ is only off by at most  $\frac{p-2}{p^3}$.

\medskip

As simple consequences of Theorem \ref{thm-p-group}, we get the following corollaries.

\begin{corollary}\label{cor-p-group} Let $G$ be a finite $p$-group with  $p\geq 2$ prime.
 If $G$ is non-cyclic, then $$\mathsf D(G)\leq \frac{2p-1}{p^2}|G|.$$
\end{corollary}

\begin{proof}
If $G$ is abelian, then, since $G$ is a non-cyclic $p$-group, there must be a subgroup $H\leq G$ with $H\cong C_p^2$. Then, from \cite[Theorem 3.2]{DavI} and \eqref{eta-p2},  we obtain $$\mathsf D(G)\leq \mathsf D(H)|G/H|=\mathsf D(C_p^2)|G/H|=(2p-1)\frac{1}{p^2}|G|,$$ as desired. On the other hand, if $G$ is non-abelian, then Theorem \ref{thm-p-group} yields $$\mathsf D(G)\leq \frac{p^2+2p-2}{p^3}|G|\leq \frac{2p-1}{p^2}|G|,$$ completing the proof.
\end{proof}

\begin{corollary}\label{cor-nilpotent} Let $G$ be a finite nilpotent group.
 If $G$ is non-abelian, then $$\mathsf D(G)\leq \frac{p^2+2p-2}{p^3}|G|,$$ where $p$ is the smallest prime divisor of $|G|$.
\end{corollary}

\begin{proof}
A finite nilpotent group is a direct product of its Sylow subgroups \cite[Theorem 5.2.4]{Robinson}. Thus, if every Sylow subgroup were abelian, then $G$ would be abelian, contrary to hypothesis. As result, we conclude that $G$ has a non-abelian Sylow $q$-group $P\leq G$ for some prime $q\mid |G|$. But then \cite[Theorem 3.2]{DavI} and Theorem \ref{thm-p-group} give $$\mathsf D(G)\leq \mathsf D(P)|G/P|\leq \frac{q^2+2q-2}{q^3}|P||G/P|=\frac{q^2+2q-2}{q^3}|G|\leq \frac{p^2+2p-2}{p^3}|G|,$$ as desired.
\end{proof}

\section{The Non-Abelian group of Order $pq$}\label{sec-pq-group}

All groups of order $p^2$, where $p$ is prime, are abelian \cite[Theorem 1.6.15]{Robinson}.
A non-abelian group of order $pq$, where $p$ and $q$ are distinct primes with $p<q$, exists precisely when $p\mid q-1$ and, in such case, is unique (up to isomorphism), being given by the presentation \cite[Theorem 3.4.4]{Jo-Kw-Xu13a}
$$F_{pq}:=\la\alpha,\,\tau:\; \alpha^q=1,\quad\tau^p=1,\quad \alpha\tau=\tau\alpha^r\ra,$$ where $r\in \Z$ is an integer such that \be\label{pq-r-def}r^p\equiv 1\mod q\quad\mbox{ but }\quad r\not\equiv 1\mod q.\ee Note this means that the multiplicative order of $r$ modulo $q$ is equal to $p$. Since all proper subgroups of $F_{pq}$ are of prime order, they are cyclic, which makes $F_{pq}$ an example of a non-abelian group having all proper subgroups cyclic.

 In Section \ref{sec-genbounds}, we will be able to reduce the question of bounding $\mathsf D(G)$,
for more arbitrary $G$, to the case of $G=F_{pq}$ and one other group (treated in Section \ref{sec-near-dihedral}).
This makes determining $\mathsf D(F_{pq})$ fairly important,
which will be accomplished in the main result of this section, Theorem \ref{thm-pq-main}. The proof of Theorem \ref{thm-pq-main} will be divided into several lemmas.


\begin{theorem}\label{thm-pq-main} Let  $p$ and $q$ be  primes with $p\mid q-1$. Then $$\mathsf D(F_{pq})=2q.$$
\end{theorem}

Let us begin first with the lower bound.

\begin{lemma}\label{pqlem-lower-bound}  Let  $p$ and $q$ be  primes with $p\mid q-1$. Then $$\mathsf D(F_{pq})\geq 2q.$$
\end{lemma}

\begin{proof} Let $G=F_{pq}$.
Consider the  sequence $$S=\tau^{p-1}\bdot \alpha^{[q-1]}\bdot \tau\alpha^{r+1}\bdot \alpha^{[q-1]}\in \Fc(G).$$ Since $$\tau^{p-1} \alpha^{q-1} \tau\alpha^{r+1} \alpha^{q-1}=\tau^{p-1}\tau\alpha^{-r}\alpha^{r+1}\alpha^{q-1}=\tau^p\alpha^q=1,$$ we see that $S$ is a product-one sequence. We claim that $S$ is an atom, which will show $\mathsf D(G)\geq |S|=2q$, as desired. Assuming to the contrary that $S$ is not an atom, we obtain a factorization $S=T_1\bdot T_2$ with $T_1,\,T_2\in \Fc(G)$ both nontrivial, product-one sequences. Clearly, either $T_1=\alpha^{[q]}$ or $T_2=\alpha^{[q]}$, say $T_1$, and then  $T_2=\tau^{p-1}\bdot \tau\alpha^{r+1}\bdot \alpha^{[q-2]}$. Since $T_2$ has product-one, it follows in view of \cite[Lemma 2.3]{DavI} that $$1=\tau^{p-1}\alpha^x(\tau\alpha^{r+1})\alpha^{q-2-x}=\alpha^{(x+1)(r-1)},\quad\mbox{ for some $x\in [0,q-2]$}.$$ Since $\ord(\alpha)=q$ is prime, this means $x+1\equiv 0\mod q$ or $r-1\equiv 0\mod q$. The latter is ruled out by \eqref{pq-r-def} while the former is impossible in view of $x\in [0,q-2]$, yielding the desired contradiction.
\end{proof}

For the proof of Theorem \ref{thm-pq-main}, we will need to adapt the ideas from Section \ref{sec-G'} using very specific knowledge about the conjugacy  structure of $F_{pq}$. To this end, we summarize some easily verified group theoretic properties for $G=F_{pq}$:
\ber\label{fact-G'}&&G'=[G,G]=\la \alpha\ra\cong C_q\quad\und \quad \mathsf Z(G)=\{1\};\\
&&\mathsf C_G(g)=\la g\ra\quad\mbox{ for every $g\in G\setminus \{1\}$};\label{fact-C_G}\\
&& \ord(g)=q\quad\mbox{ for every $g\in G'\setminus \{1\}$}\quad\und\quad \ord(g)=p\quad\mbox{ for every $g\in G\setminus G'$}\label{fact-order};\eer
and the conjugacy classes of $G$ are given by \be \label{conj-classes} \{1\},\quad \{\alpha^x,\alpha^{xr},\alpha^{xr^2},\ldots,\alpha^{xr^{p-1}}\}\quad\mbox{ for $x\in X$},\quad \und
\quad \tau^{y}\la \alpha\ra \quad\mbox{ for $y=1,2,\ldots,p-1$},\ee where $X\subseteq [1,q-1]$ is some subset of size $|X|=\frac{q-1}{p}$. We continue with a simple lemma.

\begin{lemma}
\label{pq-lem-help1} Let  $p$ and $q$ be  primes with $p\mid q-1$,  let $G=F_{pq}$, let  $S\in \mathcal F(G'\setminus \{1\})$ and let $x\in G\setminus G'$. Then $$|\pi(x\bdot S)|\geq \min \{q,\,|x\bdot S|\}.$$
\end{lemma}

\begin{proof}
We may w.l.o.g. assume $|S|\leq q-1$, for if $|S|\geq q$, then applying the lemma to any length $q-1$ subsequence of $S$ completes the proof. We need to show \be\label{firstdown}|\pi(x\bdot S)|\geq |x\bdot S|.\ee
If $S$ is the empty sequence, then \eqref{firstdown} is trivial, so we assume $|S|\geq 1$ and proceed by induction on $|S|\leq q-1$. Let $y\in \supp(S)$ and set $S'=y^{[-1]}\bdot S$. Since $S\in \mathcal F(G'\setminus \{1\})$, we have \be\label{ribble}\la y\ra =G'\cong C_q.\ee Since $x\in G\setminus G'$ and $\supp(S')\subseteq \supp(S)\subseteq G'$, it follows that \be\label{seconddown}\pi(x\bdot S')\subseteq xG'\neq G'.\ee Note that \be\label{thirddown}\pi(x\bdot S')y\cup y\pi(x\bdot S')\subseteq \pi(x\bdot S'\bdot y)=\pi(x\bdot S).\ee
By induction hypothesis, $|\pi(x\bdot S')|\geq |x\bdot S'|=|x\bdot S|-1$. Thus $|\pi(x\bdot S)|\geq |x\bdot S|$ follows from \eqref{thirddown}, completing the proof, unless $\pi(x\bdot S')y= y\pi(x\bdot S')$. However, this is equivalent to saying $$y^{-1}\pi(x\bdot S')y=\pi(x\bdot S').$$ Thus the set $\pi(x\bdot S')$ must be a union of orbits under the action of conjugation by elements from $\la y\ra=G'$ (in view of \eqref{ribble}).
In particular, the $G'$-orbit of $z$ is contained in $\pi(x\bdot S')$ for any $z\in \pi(x\bdot S')$. By \eqref{seconddown}, we have $z\in G\setminus G'$ for any such $z\in \pi(x\bdot S')$. But since $\mathsf C_G(z)\cap G'=\la z\ra\cap G'=\{1\}$, and since the size of the $G'$-orbit containing $z$ is $|G'|/(\mathsf C_G(z)\cap G')=|G'|=q$ (by \eqref{orbit-size}), it follows that $$|\pi (x\bdot S)|\geq |\pi(x\bdot S')|\geq q\geq |x\bdot S|,$$ completing the proof.
\end{proof}

The next lemma improves the bound from Lemma \ref{pq-lem-help1} under some mild restrictions and requires  a more technical argument.

\begin{lemma}
\label{pq-lem-help2} Let  $p$ and $q$ be  primes with $p\mid q-1$,  let $G=F_{pq}$, let  $S\in \mathcal F(G'\setminus \{1\})$ and let $g_1,\,g_2\in G\setminus G'$. Suppose $g_1 g_2\notin G'$. Then $$|\pi(g_1\bdot g_2\bdot S)|\geq \min \{q,\,2|S|+1\}.$$
\end{lemma}

\begin{proof}
Let $g_1=\tau^{x_1}\alpha^{a_1}$, let $g_2=\tau^{x_2}\alpha^{a_2}$, and let $S=\alpha^{y_1}\bdot\ldots\bdot \alpha^{y_\ell}$, where $\ell=|S|$.
Since $S\in \Fc(G'\setminus \{1\})$, we have \be\label{eq3} y_i\not\equiv 0\mod q\quad\mbox{for all $i\in[1,\ell]$},\ee
since $g_1,\,g_2\in G\setminus G'$, we have \be\label{eq1}x_1\not\equiv 0\mod p\quad\und\quad x_2\not\equiv 0\mod p,\ee and since $g_1g_2\notin G'$, we have \be\label{eq2}x_1+x_2\not\equiv 0\mod p.\ee Since the multiplicative order of $r$ modulo $q$ is $p$ (care of  \eqref{pq-r-def}), we deduce from  \eqref{eq1} and \eqref{eq2} that $$\{1,r^{x_2},r^{x_1+x_2}\}$$ is a set of $3$ distinct non-zero residue classes modulo $q$.

Now $$\pi(g_1\bdot g_2\bdot S)=\{\pi(T^*):\;T^*\in \Fc^*(G) \quad \und\quad [T^*]=g_1\bdot g_2\bdot S\}.$$  Every $T^*\in \Fc^*(G)$ with $[T^*]=g_1\bdot g_2\bdot S$ has the term equal to $g_1$ either preceding or following the term equal to $g_2$. Consider only those $T^*\in \Fc^*(G)$ with $[T^*]=g_1\bdot g_2\bdot S$ such that $g_1$ precedes  $g_2$. Then each term $\alpha^{y_i}$ of $S$ can either occur before $g_1$ in $T^*$, between $g_1$ and $g_2$, or after $g_2$. Furthermore, $$\pi(T^*)=\tau^{x_1+x_2}\alpha^{a_1r^{x_2}+a_2+\Sum{i=1}{\ell}y_iw_i},$$ where $w_i\in \{r^{x_1+x_2},\,r^{x_2},\,1\}$ is dependent on whether the term $\alpha^{y_i}$ of $S$ occurs before $g_1$ in $T^*$, between $g_1$ and $g_2$, or after $g_2$. Combining these thoughts,  we find that
\be\label{maybel}|\pi(g_1\bdot g_2 \bdot S)|\geq \left|\left\{\phi_{q\Z}(y):\; y\in  \Sum{i=1}{\ell}y_i\{r^{x_1+x_2},\,r^{x_2},\,1\}\right\}\right|.\ee The right hand side of \eqref{maybel} is just the number of distinct residue classes modulo $q$ contained in the integer sumset from \eqref{maybel}. We showed above that $\{r^{x_1+x_2},\,r^{x_2},\,1\}$ is a set of $3$ distinct residue classes modulo $q$, and  since \eqref{eq3} ensures that each $y_i\not\equiv 0\mod q$, it follows that each summand $y_i\{r^{x_1+x_2},\,r^{x_2},\,1\}$ in the sumset from \eqref{maybel} has size $3$ modulo the prime $q$. Thus, applying the Cauchy-Davenport Theorem to \eqref{maybel} yields $$|\pi(g_1\bdot g_2 \bdot S)|\geq \min\{q,\,2\ell+1\}.$$ Since $\ell=|S|$, the proof is now complete.
\end{proof}

The following lemma will be quite helpful.

\begin{lemma}\label{pq-lem-porbit} Let  $p$ and $q$ be  primes with $p\mid q-1$,  let $G=F_{pq}$, and let  $S\in \mathcal F(G\setminus\{1\})$. If $\la \supp(S)\ra=G$, then $$|\pi(S)|\geq \min \{p,\,|S|\}.$$
\end{lemma}

\begin{proof}
We may w.l.o.g. assume $|S|\leq p$, for if $|S|>p$, then applying the lemma to any length $p$
subsequence of $S$ that generates $G$ completes the proof (note any $2$ non-commuting terms generate $G$). We need to show $|\pi(S)|\geq |S|$.
Factor $S=S_{G'}\bdot S_{G\setminus G'}$ with $S_{G'}\mid S$ the subsequence consisting of
all terms from $G'$. Note, since $S\in \mathcal F(G\setminus\{1\})$, that no term of $S$ is
equal to $1$. In view of $\la \supp(S)\ra=G$, there must be some $g_0\in \supp(S)$ with
$g_0\notin G'$. From Lemma \ref{pq-lem-help1}, we have
\be\label{firstdown2}|\pi(g_0\bdot S_{G'})|\geq |g_0\bdot S_{G'}|.\ee
Let $R\mid S$ be a maximal length subsequence such that $g_0\bdot S_{G'}\mid R$ and $|\pi(R)|\geq |R|$. Note that $R$ exists in view of \eqref{firstdown2}.  If $R=S$, then the proof is complete, so assume otherwise and let $x\in \supp(S\bdot R^{[-1]})$. Since  $S_{G'}\mid R$, we have \be\label{whoof2}x\in G\setminus G'.\ee Since $g_0\bdot S_{G'}\mid R$, we could only have $|R|=1$ if $|S_{G'}|=0$ and $\supp(S)\subseteq \mathsf C_G(g_0)=\la g_0\ra$ (in view of \eqref{fact-C_G}). However, $\supp(S)\subseteq\la g_0\ra$ would contradict the hypothesis $\la \supp(S)\ra=G$. Therefore we conclude that $|R|\geq 2$. Hence \be\label{whoof}|\pi(R)|\geq |R|\geq 2.\ee

We have $\pi(R)x\cup x\pi(R)\subseteq \pi(R\bdot x)$. Thus $|\pi(R\bdot x)|\geq |R\bdot x|$ will follow, contradicting the maximality of $R$, unless $\pi(R)x= x\pi(R)$, which is equivalent to saying $$x^{-1}\pi(R)x=\pi(R).$$ In consequence, $\pi(R)$ must be a union of orbits under the action of conjugation by elements from $\la x\ra$.


Let $g\in G$ be an arbitrary element. Then $g$ is contained in a $\la x\ra$-orbit of size $|\la x\ra|/|\mathsf C_G(g)\cap \la x\ra|$ (cf. \eqref{orbit-size}). In particular, in view of \eqref{fact-C_G}, \eqref{fact-order} and \eqref{whoof2}, we see that the size of the $\la x\ra$-orbit containing $g$ is either $1$ (if $g\in \la x\ra$) or $p$ (otherwise). Thus, if $\pi(R)$ contains some element from $G\setminus \la x\ra$, then (as noted above) it will contain the full $\la x\ra$-orbit containing this element, implying that $|\pi(R\bdot x)|\geq |\pi(R)|\geq p\geq |S|\geq |R\bdot x|$, which would contradict the maximality of $R$. So we instead conclude that \be\label{forcit}\pi(R)\subseteq \la x\ra.\ee Since $x\in G\setminus G'$ (in view of \eqref{whoof2}), it is readily seen that each element of $\la x\ra$ is from a separate $G'$-coset. However, as noted in Section \ref{sec-notation}, the set $\pi(R)$ is contained in a single $G'$-coset. Thus
\eqref{forcit} ensures that $|\pi(R)|\leq |G'\cap \la x\ra|=1$, contradicting \eqref{whoof} to complete the proof.
\end{proof}

The following lemma shows that a sufficiently long sequence having a product in $G'$ must actually have a product-one subsequence.

\begin{lemma}\label{pq-lem-dG}  Let  $p$ and $q$ be  primes with $p\mid q-1$,  let $G=F_{pq}$, and let  $S\in \mathcal F(G)$. If $\pi(S)\subseteq G'$ and $|S|\geq q$, then $1\in \Pi(S)$.
\end{lemma}

\begin{proof} By hypothesis, $\phi_{G'}(S)\in \Fc(G/G')\cong \Fc(C_p)$ is a product-one sequence. Let $S=T_1\bdot\ldots \bdot T_\ell$ be a factorization of $S$ with $\phi_{G'}(T_i)\in \mathcal A(G/G')$ for $i=1,\ldots,\ell$. Note (in view of \cite[Lemma 2.4]{DavI}) that \be\label{taffymal}1\leq |T_i|\leq \mathsf D(G/G')=\mathsf D(C_p)=p\quad\mbox{ for $i\in [1,\ell]$}.\ee  If $1\in \pi(T_i)$, then the lemma is complete in view of $\pi(T_i)\subseteq \Pi(S)$. Therefore we may assume $1\notin \pi(T_i)$ for every $i\in [1,\ell]$.

Observe that \be\label{static}\pi(T_1)\Big(\{1\}\cup \pi(T_2)\Big)\ldots \Big(\{1\}\cup \pi(T_\ell)\Big)\subseteq \Pi(S).\ee Since $1\notin \pi(T_i)$ for each $i$, we have \be\label{waqy1}|\{1\}\cup \pi(T_i)|=|\pi(T_i)|+1\quad\mbox{ for $i\in[1,\ell]$}\ee As remarked in Section \ref{sec-notation}, each $\pi(T_i)$ is contained in a single $G'$-coset, which must be $G'$ itself in view of $\phi_{G'}(T_i)\in \mathcal A(G/G')$. Thus \be\label{waqy2}\{1\}\cup \pi(T_i)\subseteq G'\quad\mbox{for $i\in [1,\ell]$}.\ee

Next, we proceed to show that   \be|\pi(T_i)|\geq |T_i|\quad\mbox{ for $i\in [1,\ell]$}.\label{waqy3}\ee Let $i\in [1,\ell]$ be arbitrary. If $\supp(T_i)\cap G'\neq \emptyset$, then $\phi_{G'}(T_i)\in \mathcal A(G/G')$ forces $|T_i|=1$, in which case \eqref{waqy3} is clear. If  $\supp(T_i)\subseteq G\setminus G'$ but $\supp(T_i)\nsubseteq H$ for every $H\leq G$ with $|H|=p$, then $\la \supp(T_i)\ra=G$. In this case, since \eqref{taffymal} ensures  $|T_i|\leq p$, Lemma \ref{pq-lem-porbit} gives \eqref{waqy3}. Finally, consider the case when  $\supp(T_i)\subseteq H$ for some $H\leq G$ with $|H|=p$.  In this case, $\pi(T_i)\subseteq H$, which combined with \eqref{waqy2} gives $\pi(T_i)\subseteq H\cap G'=\{1\}$. Hence $\pi(T_i)=\{1\}$, contrary to \eqref{waqy1}. Thus \eqref{waqy3} is established in all cases.

In view of \eqref{waqy2} and $G'=\la \alpha\ra\cong C_q$, we can apply the Cauchy-Davenport Theorem to the product-set from \eqref{static}, yielding \ber\nn|\Pi(S)\cap G'|&\geq& \min\{q,\,|\pi(T_1)|+\Sum{i=2}{\ell}|\{1\}\cup \pi(T_i)|-\ell+1\}\\&=&\min\{q,\,\Sum{i=1}{\ell}|\pi(T_i)|\}
\geq  \min\{q,\,\Sum{i=1}{\ell}|T_i|\}\label{goobl}
=\min\{q,\,|S|\}=q,\eer where the first equality is from \eqref{waqy1}, the second inequality is from \eqref{waqy3}, the second equality is from $S=T_1\bdot \ldots\bdot T_\ell$ being a a factorization of $S$, and the final equality is in view of the hypothesis $|S|\geq q$. In view of \eqref{goobl} and $|G'|=q$, it follows that $1\in G'=\Pi(S')\cap G'$, completing the proof.
\end{proof}

It is now a simple corollary to determine the small Davenport constant of $F_{pq}$, which was first achieved by Bass \cite{Ba07b}.

\begin{corollary}\label{pq-cor-small-d}
Let  $p$ and $q$ be  primes with $p\mid q-1$. Then $$\mathsf d(F_{pq})=q+p-2.$$
\end{corollary}

\begin{proof}
The sequence $\alpha^{[q-1]}\bdot \tau^{[p-1]}\in \Fc(F_{pq})$ is readily seen  to have no nontrivial, product-one subsequence. Thus $\mathsf d(F_{pq})\geq p+q-2$. To show $\mathsf d(F_{pq})\leq p+q-2$, let $S\in \Fc(F_{pq})$ be a sequence with $|S|\geq q+p-1$. We need to show $1\in \Pi(S)$. In view of Lemma \ref{pq-lem-dG}, to show $1\in \Pi(S)$, it suffices to show $S$ has a subsequence $T\mid S$ with $|T|\geq q$ and $\pi(T)\subseteq G'$. However, this is equivalent to finding a product-one subsequence of $\phi_{G'}(S)$ having length at least $q$. Repeated application of the definition of $\mathsf D(G/G')=\mathsf D(C_p)=p$ (in view of \cite[Lemma 2.4]{DavI}) to $\phi_{G'}(S)$ gives a product-one subsequence of $\phi_{G'}(S)$ with length at least $|S|-\mathsf D(G/G')+1=|S|-p+1\geq q$, with the final inequality in view of the hypothesis $|S|\geq q+p-1$. Thus the proof is complete.
\end{proof}

If $|S|\geq \mathsf d(F_{pq})+1$, then we are guaranteed a nontrivial, product-one subsequence $T\mid S$ but know nothing about its length apart from the trivial bound $|T|\leq \mathsf d(F_{pq})+1$.  Lemma \ref{pq-lem-short-zs} shows that when $|S|$ is slightly larger than $\mathsf d(F_{pq})+1$, then we can be assured of finding a nontrivial, product-one subsequence of length at most $q$.

\begin{lemma}\label{pq-lem-short-zs}
Let  $p$ and $q$ be  primes with $p\mid q-1$,  let $G=F_{pq}$, and let  $S\in \mathcal F(G)$. If $|S|\geq q+2p-3$, then there is a nontrivial, product-one subsequence $T\mid S$ with $|T|\leq q$. In other words, $\eta(F_{pq})\leq q+2p-3$.
\end{lemma}

\begin{proof}
We handle two cases.

\subsection*{Case 1:} $\mathsf h\big(\phi_{G'}(S)\big)\leq |S|-p+1$.

We aim to show that there exists a subsequence $T'\mid S$ with
\be\label{ohyeah} |T'|=q\quad\und \quad \pi(T')\subseteq G'.\ee Once \eqref{ohyeah} is established, we can apply Lemma \ref{pq-lem-dG} to $T'$ to find a nontrivial, product-one subsequence $T\mid T'$ with $|T|\leq |T'|=q$, as desired. Thus it remains to establish \eqref{ohyeah} to complete Case 1. If $\Pi_q\big(\phi_{G'}(S)\big)=G/G'$, then \eqref{ohyeah} readily follows, completing the case. Therefore, we can assume otherwise: \be \label{zuletide}\Pi_q\big(\phi_{G'}(S)\big)\neq G/G'.\ee Thus, since $G/G'\cong C_p$ with $p$ prime, it follows that  $\mathsf H\Big(\Pi_q\big(\phi_{G'}(S)\big)\Big)=\{1\}$. Consequently,  applying \cite[Theorem 4.2]{DavI} to $\Pi_q\big(\phi_{G'}(S)\big)$ yields
\be\label{victress}|\Pi_q\big(\phi_{G'}(S)\big)|\geq \Summ{g\in G/G'}\min\{q,\vp_g\big(\phi_{G'}(S)\big)\}-q+1,\ee

If $\vp_g\big(\phi_{G'}(S)\big)\leq q$ for all $g\in G/G'$, then \eqref{victress} together with the hypothesis $|S|\geq q+2p-3$ yields $|\Pi_q\big(\phi_{G'}(S)\big)|\geq |S|-q+1\geq 2p-2\geq p$. If $\vp_g\big(\phi_{G'}(S)\big)> q$ holds for precisely one $g\in G/G'$, then \eqref{victress} together with the case hypothesis yields $|\Pi_q\big(\phi_{G'}(S)\big)|\geq (q+p-1)-q+1=p$. Finally, if $\vp_g\big(\phi_{G'}(S)\big)> q$ holds for more than one $g\in G/G'$, then \eqref{victress} yields
$|\Pi_q\big(\phi_{G'}(S)\big)|\geq 2q-q+1= q+1\geq p$. In all cases, we find that $|\Pi_q\big(\phi_{G'}(S)\big)|\geq p=|G/G'|$, which contradicts \eqref{zuletide}, completing Case 1.

\subsection*{Case 2:} $\mathsf h\big(\phi_{G'}(S)\big)\geq |S|-p+2$.

If there were at least $q$ terms of $S$ from $G'\cong C_q$, then there would be a nontrivial, product-one sequence of length at most $\mathsf d(G')+1=q$ (care of \cite[Lemma 2.4.4]{DavI}), as desired. Therefore we may assume there are at most $q-1$ terms of $S$ from $G'$.
Thus, since $|S|-p+2\geq q+p-1\geq q$, we see that the case hypothesis implies that there exists a $G'$-coset $\tau^xG'$ with $x\in [1,p-1]$ such that $\vp_{\tau^xG'}(S)\geq |S|-p+2.$ Let $S_{\tau^xG'}\mid S$ be the subsequence of all terms from $\tau^xG'$, so \be\label{justceat1}|S_{\tau^xG'}|=\vp_{\tau^xG'}(S)\geq |S|-p+2\geq q+p-1.\ee Since $x\in [1,p-1]$, each element $g\in \tau^xG'$ has $\ord(g)=p$ (care of \eqref{fact-order}). In consequence, we have \be\label{justceat2}\mathsf h(S_{\tau^xG'})\leq p-1,\ee as otherwise a subsequence of $S_{\tau^xG'}$ consisting of the same term repeated $p\leq q$ times would give the desired product-one subsequence. Since $|S_{\tau^xG'}|\geq q+p-1\geq p$, it follows from \cite[Lemma 2.6]{DavI} and \eqref{justceat2}  that there exist nonempty subsets $A_1,\ldots,A_p\subseteq G'$ with $(\tau^xA_1)\bdot\ldots\bdot (\tau^xA_p)$ a setpartition of $S_{\tau^xG}$. In particular, if $1\in (\tau^xA_1)\ldots (\tau^xA_p)$, then $S_{\tau^xG'}$ will have a product-one subsequence of length $p\leq q$, completing the proof. Thus it remains to show $1\in (\tau^xA_1)\ldots (\tau^xA_p)$ to complete the proof.

Since each $A_i\subseteq G'$ with the commutator subgroup $G'$ normal in $G$, it follows that
\be\label{hipturn}(\tau^xA_1)\ldots (\tau^xA_p)=(\tau^x)^pA'_1\ldots A'_p=A'_1\ldots A'_p\subseteq G'\ee for some subsets $A'_i\subseteq G'$ with $|A_i|=|A'_i|$ for all $i\in [1,p]$. Thus, since $G'\cong C_q$ with $q$ prime, we can invoke the Cauchy-Davenport Theorem, recall that $(\tau^xA_1)\bdot\ldots\bdot (\tau^xA_p)$ a setpartition of $S_{\tau^xG}$, and then use \eqref{justceat1}  to obtain \ber\nn|A'_1\ldots A'_p|&\geq&
\min\{q,\,\Sum{i=1}{p}|A'_i|-p+1\}=\min\{q,\,\Sum{i=1}{p}|\tau^xA_i|-p+1\}\\&=&
\min\{q,\,|S_{\tau^xG'}|-p+1\}= q=|G'|.\nn\eer As a result, the inclusion in \eqref{hipturn} must be an equality, which implies $1\in G'=(\tau^xA_1)\ldots (\tau^xA_p)$, completing the proof as mentioned above.
\end{proof}

Next, we show that a counter-example to Theorem \ref{thm-pq-main} cannot have many terms from $G'=\la \alpha\ra$.

\begin{lemma} \label{pq-lem-not-many-alpha}Let  $p$ and $q$ be  odd primes with $p\mid q-1$,  let $G=F_{pq}$, and let  $S\in \mathcal A(G)$. If $|S|\geq 2q+1$, then $$\vp_{G'}(S)=\Summ{g\in G'}\vp_g(S)\leq \frac{q-3}{2}.$$
\end{lemma}

\begin{proof}
Since $S$ has product-one, let $S^*\in \Fc^*(G)$ be an ordering of $S$, so $[S^*]=S$, with $\pi(S^*)=1$. If $\supp(S)\subseteq G'$, then $|S|\leq \mathsf D(G')=\mathsf D(\la \alpha\ra)=q$ (in view of $S\in \mathcal A(G)$ and \cite[Lemma 2.4.4]{DavI}), contradicting that $|S|\geq 2q+1$. Thus $S$ must have a term from $G\setminus G'$, and in view of \cite[Lemma 2.3]{DavI}, we can assume the first term of $S^*$ is from $G\setminus G'$.

Suppose $\vp_{G\setminus G'}(S)\leq 2$. Then there will be at least $|S|-2\geq 2q-1$ terms of $S$ from $G'$. But now, since the first term of $S^*$ is from $G\setminus G'$,  the pigeonhole principle guarantees that there is a consecutive subsequence $T^*\mid S^*$ with $|T^*|\geq q=|G'|$ and $\supp(T^*)\subseteq G'$. Applying \cite[Lemma 2.4.1]{DavI} to $T^*\in \Fc(G')$, we obtain a nontrivial, product-one consecutive subsequence in $S^*$ of length at most $q<|S|$, which contradicts
\cite[Lemma 2.1]{DavI}. So we instead conclude that \be\label{lotsoftau}\vp_{G\setminus G'}(S)\geq 3.\ee

We claim that \eqref{lotsoftau} implies there is a subsequence $g_1\bdot g_2\mid S$ with \be\label{stitic}g_1,\,g_2\in G\setminus G'\quad\und\quad g_1g_2\notin G'.\ee To see this, in view of \eqref{lotsoftau}, let $x,y,z\in \supp(S)$ be terms with $x,\,y,\,z\in G\setminus G'$ and $x\bdot y\bdot z\mid S$ and assume by contradiction that
 $xy,\,xz,\,yz\in G'$. Then $\phi_{G'}(x)\phi_{G'}(y)=\phi_{G'}(x)\phi_{G'}(z)=1$, implying $\phi_{G'}(y)=\phi_{G'}(z)$. But now $yz\in G'$ implies  $1=\phi_{G'}(y)\phi_{G'}(z)=\phi_{G'}(y)^2$, so that $\ord(\phi_{G'}(y))\mid 2$. However, since $G/G'\cong C_p$ with $p$ odd by hypothesis, $\ord(\phi_{G'}(y))$ cannot be even, forcing $\ord(\phi_{G'}(y))=1$. Thus $y\in G'$, contrary to its definition. This establishes \eqref{stitic}, as claimed.

Assume by contradiction that $$\vp_{G'}(S)=\Summ{g\in G'}\vp_g(S)\geq \frac{q-1}{2}$$ and let $T\mid S$ be a subsequence with $\supp(T)\subseteq G'$ and $|T|= \frac{q-1}{2}$. Since $S$ is an atom of length $|S|\geq 2q+1\geq 2$, we have $\supp(T)\subseteq \supp(S)\subseteq G\setminus \{1\}$. Thus we can apply Lemma \ref{pq-lem-help2} using the sequence $g_1\bdot g_2\bdot T$ and thereby find that \be\label{longlong}|\pi(g_1\bdot g_2\bdot T)|\geq \min\{q,\,2|T|+1\}=q,\ee where the final inequality follows in view of $|T|= \frac{q-1}{2}$.

Since $|S|\geq 2q+1$ and $|T|= \frac{q-1}{2}$, it follows that  \be\label{grilltell}|S\bdot (T\bdot g_1\bdot g_2)^{[-1]}|=|S|-|T|-2\geq \frac{3q-1}{2}.\ee Since $p\mid q-1$ with $p$ and $q$ odd, we have $q\geq 2p+1$. Combining this with \eqref{grilltell} yields $$|S\bdot (T\bdot g_1\bdot g_2)^{[-1]}|\geq q+\frac{q-1}{2}\geq q+p=\mathsf d(G)+2,$$ where the final inequality follows from Corollary \ref{pq-cor-small-d}. Thus applying the definition of $\mathsf d(G)$ to the sequence $S\bdot (T\bdot g_1\bdot g_2)^{[-1]}$, we find a nontrivial, product-one subsequence $R\mid S$ such that $T\bdot g_1\bdot g_2\mid R^{[-1]}\bdot S$. As noted in Section \ref{sec-notation}, $\pi(R^{[-1]}\bdot S)$ is contained in a $G'$-coset. By \cite[Lemma 2.2]{DavI}, this $G'$-coset is actually the subgroup $G'$ itself.
Moreover, in view of $T\bdot g_1\bdot g_2\mid R^{[-1]}\bdot S$ and \eqref{longlong}, we see that, in fact, $\pi(R^{[-1]}\bdot S)=G'$. In particular, $1\in G'\in \pi(R^{[-1]}\bdot S)$.  As a result, $S=R\bdot (R^{[-1]}\bdot S)$ is a factorization of $S$ into $2$ nontrivial, product-one subsequences, contradicting that $S\in \mathcal A(G)$ is an atom and completing the proof.
\end{proof}

As we will see in the proof, the following lemma is essentially just a consequence of the fact that a set in $\mathbb F_q$ having multiplicative stabilizer of size at least $3$ cannot be an arithmetic progression apart from trivial extremes for its cardinality. Note, since $A\setminus \{0\}$ is a disjoint  union of sets of size $p\geq 3$ (in view of the sets from Lemma \ref{pq-lem-not-arith} being orbits under the multiplication by $r$ action), that the hypothesis $2\leq |A|\leq q-2$ in Lemma \ref{pq-lem-not-arith} actually implies $3\leq p\leq |A|\leq q-p\leq q-3$.

\begin{lemma}\label{pq-lem-not-arith}Let  $p$ and $q$ be  odd primes with $p\mid q-1$,  let $r\in \mathbb F_q^\times$ be an element of multiplicative order $p$, and let $A\subseteq \mathbb  F_q$ be a subset which is a union of sets of the form
\be\label{add-conj}\{0\}\quad \und\quad g\{1,r,r^2,\ldots,r^{p-1}\}\quad\mbox{ for $g\in \mathbb F_q\setminus \{0\}$}.\ee
If $2\leq |A|\leq q-2$, then $A$ is not an arithmetic progression.
\end{lemma}

\begin{proof} Since $p$ and $q$ are odd primes, we have $p,\,q\geq 3$. Thus,
since $r\in \mathbb F_q^\times$ has multiplicative order $p\geq 3$, we see that \be\label{dancerance}r\notin\{-1,0,1\}\quad\mbox{ with }\quad r^p=1.\ee  Let $P=\{1,r,r^2,\ldots,r^{p-1}\}$ and note that $rP=P$ in view of $r^p=1$. Now $r\{0\}=\{0\}$ and $rgP=grP=gP$ for all $g\in \mathbb F_q\setminus \{0\}$. Thus $A$ is a union of sets which are stable under multiplication by $r$, which implies that $A$ is stable under multiplication by $r$: $$A=rA.$$


Assume by contradiction that $A$ is an arithmetic progression, so $A=\{a,a+d,\ldots,a+\ell d\}$ for some $a\in \mathbb F_q$ and $d\in \mathbb F_q^\times$, where $\ell=|A|-1$. Then $A=rA=\{ra,ra+rd,\ldots,ra+\ell rd\}$ is also an arithmetic progression with difference $rd\in \mathbb F_q^\times$. However, since $2\leq |A|\leq q-2$, it is well-known (and easily shown) that the difference $d$ of the arithmetic progression $A$ is unique up to sign. Hence $rd=\pm d$, implying $r\in\{-1,1\}$, contrary to \eqref{dancerance}.
\end{proof}

The following lemma will be used in conjunction with Lemma \ref{lem-algorithm}.

\begin{lemma}\label{pq-lem-alg-help}Let  $p$ and $q$ be  odd primes with $p\mid q-1$,  let $G=F_{pq}$, and let $T_1,\ldots,T_r\in \Fc(G)$ be sequences for which \eqref{alg-conc-2} holds.
Then the following hold.
\begin{itemize}
\item[(i)] $|\pi(T_1)\ldots \pi(T_r)|\geq \min\{q-1,\,\Sum{i=1}{r}|\pi(T_i)|\}\geq \min\{q-1,\,\Sum{i=1}{r}|T_i|\}.$
\item[(ii)] If
 $\Sum{i=1}{r}|T_i|\geq q+1$, then $|\pi(T_1)\ldots \pi(T_r)|=q.$
\end{itemize}
\end{lemma}

\begin{proof}  Consider an arbitrary  $j\in [1,r-1]$. Then $\pi(T_i)^G=\pi(T_i)$ for all $i\in [1,j]$, which means that each $\pi(T_i)$, for $i\in [1,j]$, is a union of $G$-orbits. It easily seen that this property is preserved by taking product-sets: Indeed, given any $x,\,y,\,g\in G$, we have  $g^{-1}xyg=g^{-1}xgy'=x'y'$ for some  $y'\in y^G$ and $x'\in x^G$, which shows that the product-set of two orbits is stable under conjugation. Consequently, $$(\pi(T_1)\ldots\pi(T_j))^G=\pi(T_1)\ldots\pi(T_j).$$ Thus the product-set $\pi(T_1)\ldots\pi(T_j)$,  for $j\in [1,r-1]$, is also a union of  $G$-orbits.
Since $\mathsf Z(G)=\{1\}$ and $|\pi(T_j)|\geq |T_j|\geq 2$, there can be at most one orbit of size $1$ contained in $\pi(T_j)$, and so there is at least one orbit of size greater than $1$ in $\pi(T_j)$, which must have  size either  $p$ or $q$. If size $q$ occurs, then we trivially have $|\pi(T_1)\ldots \pi(T_r)|\geq |\pi(T_j)|\geq q$,
 as desired. So we instead conclude that each $\pi(T_j)$, for $j\in [1,r-1]$, is a union of $G$-orbits of size $p$ possibly union $\{1\}$. Likewise, $\pi(T_1)\ldots\pi(T_j)$, for $j\in [1,r-1]$, is also a union of  $G$-orbits of size $p$ possibly union $\{1\}$. In particular,
we have $$\pi(T_i)\subseteq G'\cong C_q\quad\mbox{ for all $i\in [1,r-1]$}.$$ Thus, since $\pi(T_r)$ is contained in a $G'$-coset (as remarked in Section \ref{sec-notation}),  the Cauchy-Davenport Theorem and Vosper's Theorem can be used to estimate the product-set $\pi(T_1)\ldots \pi(T_r)$.

Let us next deduce (ii) from (i). To this end, suppose $\Sum{i=1}{r}|T_i|\geq q+1$. If $r=1$, then we have $|\pi(T_1)\ldots \pi(T_r)|=|\pi(T_1)|\geq |T_1|=\Sum{i=1}{r}|T_i|\geq q+1=|G'|+1$, which is impossible. Thus $r\geq 2$. Applying (i) to $\pi(T_1)\ldots \pi(T_{r-1})$, we find that $$|\pi(T_1)\ldots \pi(T_{r-1})|\geq \min\{q-1,\,\Sum{i=1}{r-1}|T_i|\}.$$ If $|\pi(T_1)\ldots \pi(T_{r-1})|\geq q-1$, then the Cauchy-Davenport Theorem implies $|\pi(T_1)\ldots \pi(T_{r-1})\pi(T_r)|=q$ in view of $|\pi(T_r)|\geq |T_r|\geq 2$, as desired. Thus $|\pi(T_1)\ldots \pi(T_{r-1})|\geq \Sum{i=1}{r-1}|T_i|$, and now the Cauchy-Davenport Theorem instead implies $$|\pi(T_1)\ldots \pi(T_{r-1})\pi(T_r)|\geq \min \{q, \, \Sum{i=1}{r-1}|T_i|+|T_r|-1\}=q,$$ with the final equality in view of the hypothesis $\Sum{i=1}{r}|T_i|\geq q+1$. Thus we see that (ii) follows from (i).

It remains to prove (i). Translating between the multiplicative notation of \eqref{conj-classes} and the additive notation of Lemma \ref{pq-lem-not-arith}, we see that the sets described in \eqref{add-conj} correspond to the $G$-orbits contained in $G'$ as described by  \eqref{conj-classes}. In particular,  we see that a set $X$ which is a union of $G$-orbits of size $p$ possibly union $\{1\}$ cannot be a (multiplicative) arithmetic progression unless $|X|\leq 1$ or $|X|\geq q-1$. Thus, in view of the conclusion of the first paragraph (and since $|\pi(T_1)|\geq |T_1|\geq 2$ by \eqref{alg-conc-2}), we may assume each $\pi(T_1)\ldots\pi(T_j)$, for $j\in [1,r-1]$, is not a (multiplicative) arithmetic progression, else $|\pi(T_1)\ldots\pi(T_r)|\geq q-1$ follows, as desired. But that means we can apply Vosper's Theorem to the product-sets $\Big(\pi(T_1)\ldots\pi(T_{j})\Big)\big(\pi(T_{j+1})\Big)$, for $j\in [1,r-1]$, to obtain the estimate $$|\pi(T_1)\ldots\pi(T_r)|\geq \min\{q-1,\,\Sum{i=1}{r}|\pi(T_i)|\}\geq\min\{q-1,\,\Sum{i=1}{r}|T_i|\},$$ with the second inequality in view of \eqref{alg-conc-2}, as desired.
\end{proof}

Lemma \ref{pq-lem-the-algorithm-2original} is the counterpart to Lemma \ref{pq-lem-not-many-alpha}, showing that a counter-example to Theorem \ref{thm-pq-main} cannot have too many terms from the same order $p$ subgroup $H\leq G$.

\begin{lemma}\label{pq-lem-the-algorithm-2original} Let  $p$ and $q$ be  odd primes with $p\mid q-1$,  let $G=F_{pq}$, and let  $S\in \mathcal A(G)$. If $|S|\geq 2q+1$, then $$\vp_{H}(S)=\Summ{g\in H}\vp_g(S)\leq q \quad\mbox{ for every subgroup \ $H\leq G$ \ with  \ $|H|=p$}.$$
\end{lemma}

\begin{proof}
Since $S\in \A(G)$, let $S^*\in \Fc(G)$ be an ordering of $S$, so $[S^*]=S$, with $\pi(S^*)=1$. Since $S$ is an atom of size $|S|>1$, we have $1\notin \supp(S)$. Assume by contradiction that there is an order $p$ subgroup $H\leq G$ with \be\label{lots-of-tau}\vp_H(S)\geq q+1.\ee Consequently, since $q+1\geq p+1$, we can apply Lemma \ref{lem-algorithm} to $S^*$ using $H$ with \ $\omega=q+1$, \, $\omega_H=p+1$ \ and \ $\omega_0=0$. Let ${S'}^*=T^*_1\bdot\ldots \bdot T^*_r\bdot R^*$ be the resulting factorization with all notation as given by Lemma \ref{lem-algorithm}.
Since $\pi(S^*)=1$, \eqref{alg-conc-1} ensures that \be\label{p-one}\pi({S'}^*)=1.\ee
There are three cases depending on whether (i), (ii) or (iii) holds in Lemma \ref{lem-algorithm}.

\subsection*{Case 1:}  Lemma \ref{lem-algorithm}(i) holds. Then $\Sum{i=1}{r}|T_i|\leq \omega-1=q$ with $\la\supp(R)\ra$ a proper subgroup. In view of \eqref{lots-of-tau} and $\Sum{i=1}{r}|T_i|\leq \omega-1=q$, we see that $\la \supp(R)\ra$ must contain a term from $H$. Moreover, since  $S\in \A(G)$ with $|S|\geq 2q+1\geq 2$ ensures that no term of $S$ is equal to $1$, it follows that $\supp(R)$ contains a generating element from $H$, in which case $\la \supp(R)\ra<G$ being proper forces  $\la \supp(R)\ra=H$. But now we have $|R|\geq |S|-\Sum{i=1}{r}|T_i|\geq 2q+1-q=q+1>p=|H|$, in which case we can apply \cite[Lemma 2.4.1]{DavI} to find a product-one consecutive subsequence of $R^*$ that is nontrivial and proper, which contradicts \cite[Lemma 2.1]{DavI} in view of $[{S'}^*]=S\in \mathcal A(G)$ and \eqref{p-one}.

\subsection*{Case 2:} Lemma \ref{lem-algorithm}(ii) holds. Then $\Sum{i=1}{r}|T_i|\geq \omega=q+1$ and there are at least $p+1=|H|+1$ terms of $R$ from $H$, in which case \cite[Lemma 2.4.1]{DavI} ensures that $R$ contains a nontrivial, product-one subsequence $R'\mid R$. Since $R'\mid R$, we have $T_1\bdot\ldots \bdot T_r\mid S\bdot {R'}^{[-1]}$. In consequence, since lemma \ref{pq-lem-alg-help}(ii) and $ \Sum{i=1}{r}|T_i|\geq q+1$ show that $\pi(T_1\bdot\ldots \bdot T_r)$ is a full $G'$-coset, it follows that  $\pi(S\bdot {R'}^{[-1]})$ is also a full $G'$-coset. Hence, since  \cite[Lemma 2.2]{DavI} implies  $\pi(S\bdot {R'}^{[-1]})\subseteq G'$, we conclude that $\pi(S\bdot {R'}^{[-1]})= G'$, in which case $S=(S\bdot {R'}^{[-1]})\bdot R'$ is a nontrivial factorization of $S$, contradicting that $S\in \mathcal A(G)$ is an atom.

\subsection*{Case 3:} Lemma \ref{lem-algorithm}(iii) holds. Then $\Sum{i=1}{r}|T_i|\leq \omega-1=q$ and $\vp_H(R)=p+1$. Thus \eqref{lots-of-tau} ensures that \be\label{workit}\vp_H(T_1\bdot\ldots\bdot T_r)= \vp_H(S)-p-1\geq  q-p.\ee Since $H$ is an abelian subgroup, we see that $|\pi(T_i)|\geq |T_i|\geq 2$ (from \eqref{alg-conc-2}) ensures that each $T_i$ contains some term from $G\setminus H$. Combined with \eqref{workit}, this implies \be\label{gymnatic}|T_1\bdot\ldots\bdot T_r|\geq q-p+r\geq q-p+1,\ee where $r\geq 1$ (which is equivalent to $R\neq S$)  follow in view of of $\vp_H(R)=p+1<q+1\leq \vp_H(S)$.

Since $\Sum{i=1}{r}|T_i|\leq \omega-1=q$ and $\vp_H(R)=p+1$, there are at least $$|S|-q-p-1\geq 2q+1-q-p-1=q-p\geq p-1$$ terms of $R$ from $G\setminus H$ (recall that $p\mid q-1$ with $p$ odd implies $q\geq 2p+1$). Thus we can find a subsequence $R_\alpha\mid R$ with \be\label{seesee}|R_\alpha|=p-1\quad \und\quad\supp(R_\alpha)\cap H=\emptyset.\ee Let $g_0\in \supp(R)\cap H$. Then $\la\supp(g_0\bdot R_\alpha)\ra=G$ (in view of \eqref{seesee}), in which case Lemma \ref{pq-lem-porbit} implies that \be\label{dunsteri} |\pi(g_0\bdot R_\alpha)|\geq \min\{p,\,|\pi(g_0\bdot R_\alpha)|\}=p.\ee Since $\pi(g_0\bdot R_\alpha)$ is contained inside a $G'$-coset with $G'\unlhd G$ a normal subgroup of prime order $q$, we can apply the Cauchy-Davenport Theorem  and then make use of Lemma \ref{pq-lem-alg-help}, \eqref{dunsteri} and \eqref{gymnatic} to conclude that
\ber\nn|\pi(T_1\bdot\ldots \bdot T_r)\pi(g_0\bdot R_{\alpha})|&\geq& \min\{q,\,
|\pi(T_1\bdot\ldots \bdot T_r)|+|\pi(g_0\bdot R_{\alpha})|-1\}\\ &\geq& \min\Big\{q,\,\min\{q-1,\,\Sum{i=1}{r}|T_i|\}+|\pi(g_0\bdot R_{\alpha})|-1
\Big\}\nn\\ &\geq & \min\Big\{q,\,\min\{q-1,\,q-p+1\}+p-1
\Big\}
=q.\label{luminous}\eer

Since $\vp_H(R)=p+1$ and $\vp_H(g_0\bdot R_\alpha)=1$, we still have $p=|H|$ terms of $R\bdot (g_0\bdot R_{\alpha})^{[-1]}$ from $H$. Thus \cite[Lemma 2.4.1]{DavI} ensures that we have a nontrivial, product-one subsequence $R'\mid R\bdot (g_0\bdot R_{\alpha})^{[-1]}$. In view of \eqref{p-one} and \cite[Lemma 2.2]{DavI}, we see that $\pi(S\bdot {R'}^{-1})\subseteq G'$. However, since $S\bdot {R'}^{-1}$ contains the subsequence $T_1\bdot\ldots\bdot T_r\bdot g_0\bdot R_\alpha$, it follows from \eqref{luminous} that $\pi(S\bdot {R'}^{-1})=G'$. Thus $S=(S\bdot {R'}^{-1})\bdot R'$ is a nontrivial factorization of $S$, contradicting that $S\in \mathcal A(G)$ is an atom. This completes the proof.
\end{proof}




With the above work complete, we are now ready to begin the main portion of the proof of Theorem \ref{thm-pq-main}.

\begin{proof}[Proof of Theorem \ref{thm-pq-main}] Let $G=F_{pq}$. In view of Lemma \ref{pqlem-lower-bound}, it suffice to prove the upper bound $\mathsf D(G)\leq 2q$.
If $p=2$, then \cite[Lemma 2.4.1]{DavI} implies $\mathsf D(G)\leq |G|=2q$, as desired. Therefore we may assume $p$ is odd, and thus also $q$ in view of $p\mid q-1$. Note that this implies $$q\geq 2p+1.$$
Let $S\in \mathcal A(G)$ be an atom with $|S|=\mathsf D(G)$
and suppose by contradiction that $|S|\geq 2q+1$. Since $S\in \mathcal A(G)$ is an atom with $|S|\geq 2$, we have $1\notin \supp(S)$.  Let $S^*\in \Fc^*(G)$ be an ordering of $S$ with $\pi(S^*)=1$.
By Lemma \ref{pq-lem-not-many-alpha}, we have \be\label{little-alph} \vp_{G'}(S)\leq \frac{q-3}{2}.\ee
 We divide the proof into $2$ main cases.

\subsection*{Case 1:} $1\in \Pi_{\leq q-p}(S)$.

In view of the case hypothesis, let $U\mid S$ be a nontrivial, product-one subsequence with $|U|\leq q-p$.  Let $W=S\bdot U^{[-1]}$.

We first show that we can assume $|\la \supp(U)\ra|=p$ with $|U|\leq p$. If $\vp_{G'}(W)=0$, then set $W_0$ to be the trivial sequence. Otherwise, in view of $|W|=|S|-|U|\geq q+p+1$ and \eqref{little-alph}, we can find a subsequence $W_0\mid W$ containing all terms from $G'$ and exactly $1$ term from $G\setminus G'$. In view of Lemma \ref{pq-lem-help1}, we have $|\pi(W_0)|\geq |W_0|$; moreover, if $W_0$ is nontrivial, then $|\pi(W_0)|\geq |W_0|\geq 2$, which together with $\mathsf Z(G)=\{1\}$ ensures that $\pi(W_0)\cap \Big(G\setminus \mathsf Z(G)\Big)\neq \emptyset$. Thus, letting $W^*\in \Fc^*(G)$ be any ordering of $W$ such that $[W^*(1,|W_0|)]=W_0$, we can apply Lemma \ref{lem-algorithm} to $W^*$ taking $H$ trivial, $\omega=q+1$, \ $\omega_H=-1$, and $\omega_0=|W_0|\leq \frac{q-1}{2}$. Let ${W'}^*=T_1^*\bdot\ldots\bdot T_r^*\bdot R^*$ be the resulting factorization with all notation as given by Lemma \ref{lem-algorithm}. Since $\omega_H$ is negative, Lemma \ref{lem-algorithm}(iii) cannot hold. If Lemma \ref{lem-algorithm}(ii) holds, then Lemma \ref{pq-lem-alg-help}(ii) implies that $\pi(W)=\pi(S\bdot U^{[-1]})$ is a full  $G'$-coset. However, since $U$ is a nontrivial, product-one subsequence, \cite[Lemma 2.2]{DavI} then implies that this full $G'$-coset must be $G'$ itself, whence $S=(S\bdot U^{[-1]})\bdot U$ is a nontrivial factorization of $S$, contradicting that $S\in \mathcal A(G)$ is an atom. Therefore, we see that Lemma \ref{lem-algorithm}(i) must hold, in which case $|R|= |W|-\Sum{i=1}{r}|T_i|\geq 2q+1-(q-p)-q=p+1$ with $H:=\la\supp(R)\ra<G$ proper. Hence, since all terms of $W$ from $G'$ were included in $W_0\mid T_1$, it follows that $|H|=p$. But now we have $p+1$ terms from a group of order $p$, in which case \cite[Lemma 2.4.1]{DavI} yields a nontrivial, product-one subsequence with all terms from $H$ having length at most $p\leq q-p$. Exchanging this product-one sequence for $U$, we can now assume that \be\label{elfkin}|\la \supp(U)\ra|=p\quad \und\quad |U|\leq p\leq q-p.\ee

Let $W=S\bdot U^{[-1]}$, define $W_0\mid W$ and $W^*$ as before, and once more apply Lemma \ref{lem-algorithm} to $W^*$ taking $H$ trivial, $\omega=q+1$, \ $\omega_H=-1$, and $\omega_0=|W_0|\leq \frac{q-1}{2}$. Let ${W'}^*=T_1^*\bdot\ldots\bdot T_r^*\bdot R^*$ be the resulting factorization with all notation as given by Lemma \ref{lem-algorithm}. Repeating the above arguments using the new $U$, we again find that Lemma \ref{lem-algorithm}(i) holds  with $$|R|\geq |S|-|U|-\omega+1\geq 2q+1-p-q=q-p+1\geq p+1$$ and  $H'=\la \supp(R)\ra<G$ an order $p$ subgroup. If  $H'=H=\la\supp(U)\ra$, then all  terms from $R\bdot U$ will be from the same order $p$ subgroup. However, since $|R\bdot U|\geq  |S|-\omega+1\geq 2q+1-q=q+1$, this would contradict Lemma \ref{pq-lem-the-algorithm-2original}. Therefore, we must have $H'\neq H$.
Applying \cite[Lemma 2.4.1]{DavI} to $R$, we can find another nontrivial, product-one subsequence $U'$ satisfying \eqref{elfkin} with $\la\supp(U')\ra=H'$.

Let $V=W\bdot {U'}^{-1}\bdot U$. Thus we swap the product-one sequences $U'$ and $U$. Since $|R|\geq p+1$ with all terms from $H'$, we see that $R\bdot {U'}^{[-1]}\bdot U$ contains terms from both $H$ and $H'$. Since no term of $S$ is equal to $1$, this means that that there is a pair of non-commuting terms $g_0,\,h_0\in \supp(R\bdot {U'}^{[-1]}\bdot U)$. Consequently, if $\Sum{i=1}{r}|T_i|\geq q-1$, then Lemma \ref{pq-lem-alg-help}(i) and the Cauchy-Davenport Theorem together imply that $|\pi(T_1)\ldots \pi(T_r)\pi(g_0\bdot h_0)|\geq q$, in which case $\pi(S\bdot {U'}^{[-1]})$ is a full $G'$-coset. But then, as before, since $U'$ is a product-one subsequence, \cite[Lemma 2.2]{DavI} ensures that this $G'$-coset is $G'$ itself, so that $S=(S\bdot {U'}^{[-1]})\bdot U'$ is a nontrivial factorization of $S$, contradicting that $S\in \mathcal A(G)$ is an atom. Therefore, we must have $\Sum{i=1}{r}|T_i|\leq q-2$.

Let  $V_0=T_1\bdot \ldots \bdot T_r$ and let ${V}^*$ be an ordering of $V=W\bdot {U'}^{-1}\bdot U$ with $[{V}^*(1,|V_0|)]=V_0$. In view of Lemma \ref{pq-lem-alg-help} and $|V_0|=\Sum{i=1}{r}|T_i|\leq q-2$, we have $|\pi(V_0)|\geq |V_0|$. Thus we can once more apply Lemma \ref{lem-algorithm} to ${V}^*$ taking  taking $H$ trivial, $\omega=q+1$, \ $\omega_H=-1$, and $\omega_0=|V_0|\leq q-2$. Let ${V'}^*=T'_1\bdot \ldots \bdot T'_{r'}\bdot R'$ be the resulting factorization.
Since $V_0\mid T'_1$ with $V_0=T_1\bdot\ldots\bdot T_r$, it follows that $R'\mid R\bdot {U'}^{[-1]}\bdot U$. Now $\supp(R\bdot {U'}^{[-1]}\bdot U)\subseteq H\cup H'$ with $\vp_{H}(R\bdot {U'}^{[-1]}\bdot U)=p$. Consequently, at most $p$ terms of $R'$ are from $H$ with all other terms from $H'$. However,
as argued above, Lemma \ref{lem-algorithm}(i) must hold with all of the at least $p+1$ terms of $R'$ from the same order $p$ subgroup. Since there are only at most $p$ terms of $R'$ from $H$, this order $p$ subgroup cannot be $H$, and thus all terms of $R'$ are from $H'$ (in view of $\supp(R')\subseteq H\cup H'$). But now $\supp(R'\bdot U')\subseteq H'$ with $|R'\bdot U'|= |S|-\Sum{i=1}{r'}|T'_i|\geq |S|-\omega+1\geq 2q+1-q=q+1$ (with the first inequality from Lemma \ref{lem-algorithm}(i) and the second by hypothesis), which is contrary to Lemma \ref{pq-lem-the-algorithm-2original}. This completes Case 1.

\subsection*{Case 2:} $1\notin \Pi_{\leq q-p}(S)$.

If there were $p$ terms of $S$ from the same order $p$ subgroup, then we could apply \cite[Lemma 2.4.1]{DavI} to find a nontrivial, product-one subsequence with length at most $p\leq q-p$, which is contrary to case hypothesis. Therefore \be \label{little-tau}\vp_H(S)\leq p-1\quad\mbox{ for every }\; H\leq G\;\mbox{ with }\; |H|=p.\ee
From Lemma \ref{pq-lem-short-zs}, we can find a nontrivial, product-one subsequence $U\mid S$ with $|U|\leq q$. In view of $|S|-|U|\geq 2q+1-q=q+1\geq \frac{q-3}{2}+p$, \eqref{little-alph} and \eqref{little-tau}, we can find two non-commuting terms $g_0,\,h_0\in \supp(S\bdot U^{[-1]})\cap G\setminus G'$. Since any $2$ non-commuting terms generate $G$, we have \be\label{leedto}\la g_0,\,h_0\ra=G\quad\mbox{ with }\quad g_0,\,h_0\in G\setminus G'.\ee Let $W= (U\bdot g_0\bdot h_0)^{[-1]}\bdot S$.

If $\vp_{G'}(W)=0$, then set $W_0$ to be the trivial sequence. Otherwise, in view of $|W|=|S|-|U|-2\geq 2q+1-q-2=q-1$ and \eqref{little-alph}, we can find a subsequence $W_0\mid W$ containing all terms from $G'$ and exactly $1$ term from $G\setminus G'$. In view of \eqref{little-alph} and Lemma \ref{pq-lem-help1}, we have $|\pi(W_0)|\geq |W_0|$; moreover, if $W_0$ is nontrivial, then $|\pi(W_0)|\geq |W_0|\geq 2$, which together with $\mathsf Z(G)=\{1\}$ ensures that $\pi(W_0)\cap \Big(G\setminus \mathsf Z(G)\Big)\neq \emptyset$. Thus, letting $W^*\in \Fc^*(G)$ be any ordering of $W$ such that $[W^*(1,|W_0|)]=W_0$, we can apply Lemma \ref{lem-algorithm} to $W^*$ taking $H$ trivial, $\omega=q-p+1$, \ $\omega_H=-1$, and $\omega_0=|W_0|\leq \frac{q-1}{2}\leq q-p+1$. Let ${W'}^*=T_1^*\bdot\ldots\bdot T_r^*\bdot R^*$ be the resulting factorization with all notation as given by Lemma \ref{lem-algorithm}.
Since $\omega_H$ is negative, Lemma \ref{lem-algorithm}(iii) cannot hold. This gives two subcases based on whether (i) or (ii) from Lemma \ref{lem-algorithm} holds.

\subsection*{Case 2.1:} Lemma \ref{lem-algorithm}(ii) holds.

In this case, we have $q-p+1=\omega\leq \Sum{i=1}{r}|T_i|\leq \omega+1=q-p+2\leq q-1$. Thus Lemma \ref{pq-lem-alg-help} implies that \be\label{will2}|\pi(T_1)\ldots\pi(T_r)|\geq  \Sum{i=1}{r}|T_i|=q-p+1+\epsilon,\ee
 where $\epsilon\in \{0,1\}$.
In view of \eqref{leedto}, we have $\la\supp(R\bdot g_0\bdot h_0)\ra=G$. Since $|W|+2=|S|-|U|\geq q+1$, we have $|R|+2\geq q+1-\Sum{i=1}{r}|T_i|=p-\epsilon$. Consequently, Lemma \ref{pq-lem-porbit} implies that $|\pi(R\bdot g_0\bdot h_0)|\geq p-\epsilon$. But now the Cauchy-Davenport Theorem together with \eqref{will2} implies that \ber\label{delvid}\nn\left|\Big(\pi(T_1)\ldots\pi(T_r)\Big)\Big(\pi(R\bdot g_0\bdot h_0)\Big)\right|&\geq& \min \{q,\,|\pi(T_1)\ldots\pi(T_r)|+|\pi(R\bdot g_0\bdot h_0)|-1\}\\&\geq&\nn  \min\{q,\,(q-p+1+\epsilon)+(p-\epsilon)-1\}=q.\eer As a result, we see that $\pi(S\bdot U^{[-1]})=\pi(T_1\bdot\ldots\bdot T_r\bdot R\bdot g_0\bdot h_0)$ is a full $G'$-coset. However, since $U$ is a product-one subsequence, \cite[Lemma 2.2]{DavI} ensures that this $G'$-coset is $G'$ itself, whence $S=(S\bdot U^{[-1]})\bdot U$ is a nontrivial factorization of $S$, contradicting that $S\in \mathcal A(G)$ is an atom.

\subsection*{Case 2.2:} Lemma \ref{lem-algorithm}(i) holds.

In this case, we have $\Sum{i=1}{r}|T_i|\leq \omega-1=q-p$, so that \be\label{estimateR}|R|\geq |S|-|U|-2-(q-p)\geq 2q+1-q-2-q+p=p-1,\ee with $H:=\la\supp(R)\ra<G$ proper. As all terms of $W$ from $G'$ were included in $W_0\mid T_1$, it follows that $H$ must have order $p$. Thus \eqref{little-tau} ensures that $|R|=p-1$ with $g_0,\,h_0\in G\setminus H$. Since $|R|=p-1$, all estimates used in \eqref{estimateR} must be equalities. In particular, $\Sum{i=1}{r}|T_i|=\omega-1=q-p$.

Let $g'_0\in \supp(R)$. Since $h_0\notin H$ but $g'_0\in H$, it follows that $g'_0$ and $h_0$ are non-commuting terms from $G\setminus G'$. In particular, \eqref{leedto} holds with $g'_0$ in place of $g_0$. Let $V=W\bdot {g'}_0^{[-1]}\bdot g_0$. Thus we swap the terms $g_0$ and $g'_0$. Since  $\Sum{i=1}{r}|T_i|=\omega-1=q-p\leq q-1$, Lemma \ref{pq-lem-alg-help}(i) implies that $|\pi(V_0)|\geq |V_0|$, where $V_0=T_1\bdot\ldots\bdot T_r$. Thus, letting $V^*$ be any ordering of $V$ such that $[V^*(1,|V_0|)]=V_0$,  we can once more apply Lemma \ref{lem-algorithm} to ${V}^*$ taking  taking $H$ trivial, $\omega=q-p+1$, \ $\omega_H=-1$, and $\omega_0=|V_0|=q-p\leq \omega-1$. Let ${V'}^*={T'}^*_1\bdot \ldots \bdot {T'}^*_{r'}\bdot {R'}^*$ be the resulting factorization. As before, Lemma \ref{lem-algorithm}(iii) cannot hold, while if Lemma \ref{lem-algorithm}(ii) holds, then Case 2.1 completes the proof. Therefore, Lemma \ref{lem-algorithm}(i) must hold, in which case $\Sum{i=1}{r'}|T'_i|\leq \omega-1=q-p=|V_0|$. Since $V_0\mid T'_1$, this is only possible if $r'=1$ with $T'_1=V_0=T_1\bdot\ldots\bdot T_r$, in which case $R'=R\bdot {g'}_0^{[-1]}\bdot g_0$. However,
since $R\bdot {g'}_0^{[-1]}\bdot g_0$ contains exactly $p-2>0$ terms from $H$ along with the term $g_0\notin H$, it follows that $\la \supp(R')\ra=\la \supp(R\bdot {g'}_0^{[-1]}\bdot g_0)\ra=G$, which is contrary to Lemma \ref{lem-algorithm}(i). This completes the proof.
\end{proof}

\section{The Near Dihedral Group}\label{sec-near-dihedral}

The goal of this section is to prove the following theorem, which will be needed for the  proof of Theorem \ref{thm-3over4}. The proof uses the same strategy as for Corollary \ref{pq-cor-small-d}, though more technical care must be taken. Note, since $q$ is an odd prime possessing a square root of $-1$, that $q\equiv 1\mod 4$.

\begin{theorem}
\label{thm-extra} Let $q$ be an odd prime, let $r\in [1,q-1]$ be an integer such that $r^2\equiv -1\mod q$, and let $$G=\la \alpha,\,\tau:\; \alpha^q=1,\quad \tau^4=1,\,\quad \alpha\tau=\tau\alpha^r\ra.$$ Then $\mathsf d(G)=q+2$.
\end{theorem}

We begin first with the following analogue of Lemma \ref{pq-lem-dG}.

\begin{lemma}
\label{lem-for-small} Let $q$ be an odd prime, let $r\in [1,q-1]$ be an integer such that $r^2\equiv -1\mod q$, let $$G=\la \alpha,\,\tau:\; \alpha^q=1,\quad \tau^4=1,\,\quad \alpha\tau=\tau\alpha^r\ra,$$ and let $S\in \Fc(G)$ be a sequence such that $\phi_{G'}(S)\in\mathcal A(G/G')$, where $G'=\la \alpha\ra=[G,G]\leq G$ is the commutator subgroup. Then either $1\in \pi(S)$ or $|\pi(S)|\geq |S|$.
\end{lemma}

\begin{proof}
We begin by describing some routinely verified properties of the group $G$. First, we have
$$G'=\la \alpha\ra\quad\und\quad \mathsf Z(G)=\{1\}.$$ Apart from the subgroup $G'\leq G$, there are $q$ subgroups $H_i=\la\tau\alpha^i\ra\leq G$, for $i=0,1,\ldots, q-1$, of order $4$, which have trivial intersection with each other as well as $G'$. Each contains a single element $\tau^2\alpha^{(r+1)i}$ of order $2$, naturally generating an order $2$ subgroup contained in $H_i$. Any of the order $2$ elements along with $G'$ generates the subgroup $K=\la \alpha,\,\tau^2\ra$, which is dihedral of order $2q$. There are no other subgroups apart from $\{1\}$ and $G$. In particular, any two non-identity elements from distinct $H_i$ generate either $K$ (if both have order $2$) or $G$ (otherwise). With this information in hand, we can continue with the proof.

Since $\mathsf D(G/G')=\mathsf D(C_4)=4$ (care of \cite[Lemma 2.4.1]{DavI}) and $\phi_{G'}(S)\in
\mathcal A(G/G')$, we have $1\leq |S|\leq 4$. If $|S|=1$, then $|\pi(S)|\geq |S|$ is trivial. Therefore we may assume $\ell:=|S|\geq 2$, in which case $\supp(S)\subseteq G\setminus G'$ follows from $\phi_{G'}(S)\in
\mathcal A(G/G')$. Let $S=g_1\bdot \ldots\bdot g_\ell$ with $g_i\in G\setminus G'$.

If $|S|=2$, then $|\pi(S)|\geq 2= |S|$ follows, as desired, unless  both terms of $S$ commute. However, the only way two terms from $G\setminus G'$ can commute with each other is if they are from the same order $4$ subgroup $H_j$. However, since $H_j\cap G'=\{1\}$ for every $j\in [0,q-1]$, we see that $\pi(S)\subseteq G'\cap H_j=\{1\}$ then forces $S$ to be a product-one sequence, as desired. Therefore we may assume $|S|\geq 3$.

Observing that any two order $2$ elements have product one modulo $G'$, we see that $|S|\geq 3$
together with $\phi_{G'}(S)\in \mathcal A(G/G')$ ensures that $S$ contains at most one order $2$ element. Thus  w.l.o.g. we may assume $\ord(g_i)=4$ for $i\in [2,\ell]$, while $\ord(g_1)=2$ or $4$. Let $H_{j_i}$ be the order $4$ subgroup containing $g_i$, for $i\in [1,\ell]$. If $\supp(S)\subseteq H_{j_1}$, then, since $\pi(S')\subseteq G'$ follows in view of $\phi_{G'}(S)\in \mathcal A(G/G')$ and $\pi(S)$ being contained in a $G'$-coset (as noted in Section \ref{sec-notation}), it follows that  $\pi(S)\subseteq H_{j_1}\cap G'=\{1\}$, yielding the desired conclusion $1\in\pi(S)$. Therefore we assume there is some $g_i$ from a different order $4$ subgroup $H_{j_i}\neq H_{j_1}$, say w.l.o.g. $g_2$. But then $g_1g_2\neq g_2g_1$, so that $|\pi(g_1\bdot g_2)|=2$.

Let us show that $|\pi(g_1\bdot g_2\bdot g_3)|\geq 3$. Let $X=\pi(g_1\bdot g_2)$. Note that $g_3X\cup Xg_3\subseteq \pi(g_1\bdot g_2\bdot g_3)$.  If $g_3X\neq Xg_3$, then $|\pi(g_1\bdot g_2\bdot g_3)|\geq |X|+1=3$ follows, as claimed. Otherwise, $g_3X= Xg_3$ implies $g_3^{-1}Xg_3=X$, whence  $X$ is stable under conjugation by elements from the order $4$ subgroup $H_{j_3}=\la g_3 \ra$. Thus $|X|\geq |x^{H_{j_3}}|$ for each $x\in X$.
Since $\phi_{G'}(S)\in \mathcal A(G/G')$ is an atom with $G/G'\cong C_4$ abelian, we have $X\subseteq G\setminus G'$. By \eqref{orbit-size}, we have  $|x^{H_{j_3}}|=|H_{j_3}|/|\mathsf C_G(x)\cap H_{j_3}|$.
Now $\mathsf C_G(x)$, for $x\in G\setminus G'$, is simply equal to the order $4$ subgroup that contains $x$.  Since distinct order $4$ groups intersect trivially, it follows that $|x^{H_{j_3}}|=4$ (if $\mathsf C_G(x)\neq H_{j_3}$) or $|x^{H_{j_3}}|=1$ (if  $\mathsf C_G(x)=H_{j_3}$). If
$|x^{H_{j_3}}|=4$, then $|\pi(S)|\geq |\pi(X)|\geq 4\geq |S|$, as desired. Therefore we may assume $|x^{H_{j_3}}|=1$ for every $x\in X$, which is only possible if $X\subseteq H_{j_3}$. As noted in Section \ref{sec-notation}, we also have $\pi(g_1 \bdot g_2)=X$ contained in a $G'$-coset. Hence, since $X\subseteq H_{j_3}$ and $|H_{j_3}\cap G'|=1$, it follows that $|X|\leq 1$, which is contrary to what has already been shown. Thus $|\pi(g_1\bdot g_2\bdot g_3)|\geq 3$, as claimed.

If $|S|=3$, the proof is complete. If $|S|=4$,  repeating the above arguments using $Y=\pi(g_1\bdot g_2\bdot g_3)$ and $g_4$ in place of $X$ and $g_3$ shows that $|\pi(S)|\geq 4$, completing the proof in the final remaining case.
\end{proof}

\begin{proof}[Proof of Theorem \ref{thm-extra}]
The lower bound is easily verified by considering the sequence $\alpha^{[q-1]}\bdot \tau^{[3]}\in \Fc(G)$. It remains to prove $\mathsf d(G)\leq q+2$. Let $S\in \Fc(G)$ be a sequence with $|S|\geq q+3$. We need to show $1\in \Pi(S)$. Since $\mathsf d(G/G')+1=\mathsf D(G/G')=\mathsf D(C_4)=4$ (care of \cite[Lemma 2.4]{DavI}), repeated application of the definition of $\mathsf d(G/G')$ to $\phi_{G'}(S)$ yields a factorization $S=T_1\bdot\ldots\bdot T_\ell\bdot R$, where $\phi_{G'}(T_i)\in \mathcal A(G/G')$ for $i\in [1,\ell]$ and $|R|\leq 3$.
Since $\phi_{G'}(T_i)\in \mathcal A(G/G')$, it follows that \be\label{runran}\pi(T_i)\subseteq G'\quad\mbox{ for all $i\in [1,\ell]$}.\ee
We may assume $1\notin \pi(T_i)$ for $i\in [1,\ell]$, else the proof is complete. But then Lemma \ref{lem-for-small} implies that $$|\{1\}\cup \pi(T_i)|\geq |T_i|+1\quad\mbox{ for $i\in [1,\ell]$}.$$ Thus, since $G'\cong C_q$ is cyclic of prime order, repeated application of the Cauchy-Davenport Theorem yields
$$\Big|\pi(T_\ell)\prod_{i=1}^{\ell-1}(\{1\}\cup \pi(T_i))\Big|\geq \min\{q,\,\Sum{i=1}{\ell}|T_i|\}=\min\{q,\,|S|-|R|\}=q=|G'|,$$ where the penultimate equality follows in view of $|R|\leq 3$ and $|S|\geq q+3$. Thus, together with \eqref{runran}, we see that $1\in G'=\Pi(T_1\bdot\ldots\bdot T_\ell)\cap G'\subseteq \Pi(S)\cap G'\subseteq \Pi(S)$, as desired.
\end{proof}

\section{General Upper Bounds}\label{sec-genbounds}

The goal of this section is to give two general upper bounds for the large Davenport constant of a non-cyclic group. We begin with the first one.

\begin{theorem}\label{thm-2overp}
Let $G$ be a finite, non-cyclic group and let $p$ be the smallest prime divisor of $|G|$. Then $$\mathsf D(G)\leq \frac{2}{p}|G|.$$
\end{theorem}

\begin{proof}
In view of \cite[Theorem 3.2]{DavI}, we see that it suffices to prove $\mathsf D(H)\leq \frac{2}{p}|H|$ for any nontrivial subgroup $H\leq G$. If $G$ is abelian, then since $G$ is non-cyclic, there must be a subgroup $H\leq G$ with $H\cong C_q^2$ for some prime $q\geq p$. However, \eqref{eta-p2} gives $\mathsf D(H)=\mathsf D(C_q^2)=2q-1=\frac{2q-1}{q^2}|H|<\frac{2}{p}|H|$, as desired. Therefore we may assume $G$ is non-abelian, in which case $G$ contains a minimal non-abelian subgroup.
Thus it suffices to prove the theorem for all finite minimal  non-abelian groups, so we now assume $G$ is a minimal non-abelian group (all proper subgroups are abelian).

If $G$ is a $p$-group, then Theorem \ref{thm-p-group} gives $\mathsf D(G)\leq \frac{p^2+2p-2}{p^3}|G|\leq \frac2p|G|$, also as desired. Therefore, we may assume $G$ is a minimal non-abelian group which is not a $p$--group. The finite minimal non-abelian subgroups were classified by Miller and Moreno \cite{miller-moreno}.  When such a group is not a $p$--group, its commutator subgroup $G'$ is an elementary abelian group of prime power order. Thus $G'\cong C_q^r$ for some prime $q$ and $r\geq 1$. However, if $r\geq 2$, then $G$ contains a subgroup $H\cong C_q^2$, and the desired bound $\mathsf D(H)=\mathsf D(C_q^2)=2q-1\leq\frac{2}{p}|H|$ follows as before. Therefore we may assume $G'$ is cyclic of prime order $q$. But then the classification result of Miller and Moreno tells us that $|G|=p^nq$ for some $n\geq 1$ with $p\mid q-1$. Moreover, there is exactly one such non-abelian group of order $p^nq$ (up to isomorphism), which is given by the presentation
 $$G=\la \alpha,\,\tau:\; \alpha^q=1,\quad\tau^{p^n}=1,\quad \alpha\tau=\tau\alpha^r \ra,$$
where $r^p\equiv 1\mod q$ \ but $r\not\equiv 1\mod q$. It is now routine to calculate $$\mathsf Z(G)=\la \tau^p\ra\quad\und\quad G'=\la \alpha\ra.$$ In particular, $G'\cap \mathsf Z(G)=\{1\}$. Moreover, $G/\mathsf Z(G)$ is a non-abelian group of order $pq$. Thus  \cite[Theorem 3.3]{DavI}, Theorem \ref{thm-pq-main}  and \cite[Lemma 2.4.1]{DavI} yield $$\mathsf D(G)\leq \mathsf D(G/\mathsf Z(G))\mathsf D(\mathsf Z(G))\leq \frac2p|G/\mathsf Z(G)||\mathsf Z(G)|=\frac2p|G|,$$ completing the proof.
\end{proof}

We conclude with the following result, which improves Theorem \ref{thm-2overp} for even order groups.

\begin{theorem}\label{thm-3over4}
Let $G$ be a finite group which is neither cyclic nor isomorphic to a dihedral group of order $2n$ with $n$ odd. Then $$\mathsf D(G)\leq \frac34|G|$$
\end{theorem}

\begin{proof}
If $|G|$ is odd, then Theorem \ref{thm-2overp} gives $\mathsf D(G)\leq \frac2p|G|\leq \frac23|G|<\frac34|G|$, as desired. Therefore we may assume $|G|$ is even.
As in the proof of  Theorem \ref{thm-2overp}, it suffices to prove $\mathsf D(H)\leq \frac34|H|$ for any subgroup $H\leq G$. If $G$ is abelian, then, since $G$ is not cyclic, there must be a subgroup $H\cong C_q^2$ for some prime $q\geq 2$, whence $\mathsf D(H)=\mathsf D(C_q^2)=2q-1\leq \frac34|H|$ follows from \eqref{eta-p2}. Therefore, we may assume $G$ is non-abelian.
If $G$ contains a non-cyclic Sylow subgroup $H\leq G$, then applying Theorem \ref{cor-p-group} gives $\mathsf D(H)\leq \frac34 |G|$, as desired. Therefore we may assume all Sylow subgroups are cyclic. It is well-known (see \cite[Theorem 10.1.10]{Robinson}) that a finite group $G$ having all its Sylow subgroups cyclic must have a presentation of the form
\be\label{pres}G=\la \alpha,\,\tau:\; \alpha^n=1,\quad\tau^m=1,\quad \alpha\tau=\tau\alpha^r\ra,\ee
where $\gcd(r-1,n)=\gcd(m,n)=1$, \ $r^m\equiv 1\mod n$, and $n$ is odd. As $|G|=mn$ is even, we have $m$ even.

It is routine to calculate $$G'=\la \alpha\ra.$$ Consequently, $|G'|=\frac1m|G|$, so that if $m\geq 8$, then Theorem \ref{lem-commutator} and \cite[Theorem 3.1]{DavI} give the desired bound. Therefore, recalling that $m$ is even, we find that $m\in \{2,\,4,\,6\}$. If $m=2$, then $G$ has a cyclic, index $2$ subgroup, in which case \cite[Theorem 1.1, Section 5]{DavI} gives the desired bound. It remains to consider $m\in \{4,\,6\}$.

If $m=6$, then $H=\la \alpha,\,\tau^{2}\ra$ is a subgroup of odd order $3n$. If it is cyclic, then $H$ is a cyclic, index $2$ subgroup, which is a case that has already been handled. On the other hand, if it is non-cyclic, then applying Theorem \ref{thm-2overp} to $H$ yields the desired bound. Therefore it remains to consider the case $m=4$.

Let $q\mid n$ be a prime and observe that $H=\la \alpha^{n/q},\,\tau\ra\leq G$ is a non-abelian subgroup of order $mq=4q$ having a presentation of the form \eqref{pres} with $n=q$. Since $H$ is neither cyclic nor dihedral of order $2n'$ with $n'$ odd, we see that it suffices to show the theorem holds for $H$. Thus we may w.l.o.g. $H=G$ with $n=q$ prime in \eqref{pres}.

Since $G$ is non-abelian and $r^m=r^4\equiv 1\mod q$, we see that the multiplicative order of $r$ modulo $q$ is either $2$ or $4$. If it is $2$, then $r^2\equiv 1\mod q$, in which case $\la \tau^2\alpha\ra$ is a cyclic, index $2$ subgroup, which is a case that has already been handled. Thus it remains to consider the case when $r^2\not\equiv 1\mod q$ but $r^4\equiv 1\mod q$, which is easily seen to imply, as $q$ is prime and $r^4-1=(r^2-1)(r^2+1)$, that  $$r^2\equiv -1\mod q.$$ But now Theorems \ref{thm-extra} and \ref{lem-commutator} yield the desired bound $\mathsf D(G)\leq \mathsf d(G)+2|G'|-2= q+2+2q-2=\frac34|G|$, completing the proof.
\end{proof}



\providecommand{\bysame}{\leavevmode\hbox to3em{\hrulefill}\thinspace}
\providecommand{\MR}{\relax\ifhmode\unskip\space\fi MR }
\providecommand{\MRhref}[2]{%
  \href{http://www.ams.org/mathscinet-getitem?mr=#1}{#2}
}
\providecommand{\href}[2]{#2}

\end{document}